\documentclass[11pt,reqno]{amsart}
\usepackage{mathrsfs}
\usepackage{url}
\usepackage{mathtools}
\usepackage{multirow}
\usepackage{graphicx}
\usepackage[table,xcdraw]{xcolor}
\usepackage{latexsym,epsfig,amssymb,amsmath,amsthm,color,url,bm}
\usepackage[inline,shortlabels]{enumitem}
\usepackage{hyperref}
\usepackage[foot]{amsaddr}
\usepackage{amsmath,amsbsy}
\usepackage{mwe}
\RequirePackage[numbers]{natbib}
\usepackage{mathptmx}
\usepackage[text={16cm,24cm}]{geometry}
\usepackage{amssymb,amsmath,latexsym,color,amsthm,mathpazo,palatino,bbding,setspace,datetime,breakcites,hyphenat,microtype,diagbox,tikz,subcaption,graphicx,amsbsy}
\usepackage{mathptmx}
\usepackage{hyperref,theoremref}
\usepackage[titletoc,title]{appendix}
\usepackage[inline]{enumitem}
\usepackage{float}
\usepackage{amssymb}
\restylefloat{table}
\everymath{\displaystyle}
\usepackage{pstricks}
\usepackage{accents}
\newcommand{\ignore}[1]{}

\usepackage{xcolor}

\DeclareRobustCommand{\stirling}{\genfrac\{\}{0pt}{}}
\newcommand{\ut}[1]{\underaccent{\tilde}{#1}}
\renewcommand{\vec}[1]{\ut{#1}}
\definecolor{armygreen}{rgb}{0.29, 0.8, 0.13}
\definecolor{auburn}{rgb}{0.43, 0.21, 0.1}
\definecolor{burgundy}{rgb}{0.5, 0.0, 0.13}
\definecolor{medium red}{rgb}{.490,.298,.337}
\definecolor{dark red}{rgb}{.235,.141,.161}
\hypersetup{
	colorlinks = true,
	linkcolor = {burgundy},
	citecolor = {burgundy}, 		
	linkbordercolor = {white},
}

\allowdisplaybreaks 
 \numberwithin{equation}{section}
\newtheorem{theorem}{Theorem}[section]

\newtheorem{lemma}[theorem]{Lemma}

\newtheorem{observation}[theorem]{Observation}
\newtheorem{claim}[theorem]{Claim}

\newcommand\Mycomb[2][^n]{\prescript{#1\mkern-0.5mu}{}C_{#2}}

\newcommand\Item[1][]{%
  \ifx\relax#1\relax  \item \else \item[#1] \fi
  \abovedisplayskip=0pt\abovedisplayshortskip=0pt~\vspace*{-\baselineskip}}

\theoremstyle{definition}

\theoremstyle{definition}
\newtheorem{remark}[theorem]{Remark}

\DeclareMathOperator*{\argmax}{argmax} 
\newcommand{\uti}{\undertilde{v}}
\newcommand{\bt}{\bar{t}}
\newcommand{\h}{\hat{t}}
\newcommand{\balpha}{\bar{\alpha}}
\newcommand{\bbeta}{\bar{\beta}}
\newcommand{\halpha}{\hat{\alpha}}
\newcommand{\hbeta}{\hat{\beta}}

\newcommand{\talpha}{\tilde{\alpha}}
\newcommand{\tbeta}{\tilde{\beta}}

\title[Game Theoretic Epidemic Model]{The spread of an epidemic: a game-theoretic approach}
\date{}
\author{Sayar Karmakar, Moumanti Podder, Souvik Roy, Soumyarup Sadhukhan}
\address{Sayar Karmakar, University of Florida, 230 Newell Drive, Gainesville, Florida 32603, USA.}
\address{Moumanti Podder, Indian Institute of Science Education and Research (IISER) Pune, Dr.\ Homi Bhabha Road, Pashan, Pune 411008, Maharashtra, India.}
\address{Souvik Roy, Indian Statistical Institute, 203 Barrackpore Trunk Road, Kolkata 700108, West Bengal, India.}
\address{Soumyarup Sadhukhan, Indian Institute of Technology, Kalyanpur, Kanpur, Uttar Pradesh 208016, India.}
\email{sayarkarmakar@ufl.edu}
\email{moumanti@iiserpune.ac.in}
\email{souvik.2004@gmail.com}
\email{soumyarup.sadhukhan@gmail.com}

\begin{document}
\bibliographystyle{plainnat}

\begin{abstract}
We introduce and study a model stemming from game theory for the spread of an epidemic throughout a given population. Each agent is allowed to choose an action whose value dictates to what extent they limit their social interactions, if at all. Each of them is endowed with a certain amount of immunity such that if the viral risk/exposure is more than that they get infected. We consider a discrete-time stochastic process where, at the beginning of each epoch, a randomly chosen agent is allowed to update their action, which they do with the aim of maximizing a utility function that is a function of the state in which the process is currently in. The state itself is determined by the subset of infected agents at the beginning of that epoch, and the most recent action profile of all the agents. Our main results are concerned with the limiting distributions of both the cardinality of the subset of infected agents and the action profile as time approaches infinity, considered under various settings (such as the initial action profile we begin with, the value of each agent's immunity etc.). We also provide some simulations to show that the final asymptotic distributions for the cardinality of infected set are almost always achieved within the first few epochs. 
\end{abstract}


\keywords{game theoretic model; spread of infectious diseases in networks; spread of an epidemic; utility functions}

\maketitle

\section{Introduction}\label{sec:intro}
\subsection{Overview of the paper}\label{subsec:overview} The primary motivation that fuels our work in this paper is the need to understand how an infectious disease spreads through a population comprising intelligent, pragmatically thinking individuals who decide upon their actions (such as distancing oneself from possibly infected acquaintances via voluntary confinement to one's home) on a day-to-day basis, with the aim to maximize their respective \emph{utility functions}. The key novelty of our work lies in being able to capture, via our model, the fact that the population we consider is made up of rational beings referred to as \emph{agents} or \emph{players}. We emphasize here the need for investigation in understanding the spread of a contagion through a population whose members are not just helpless entities exposed to the infection at the whim of nature alone (see \S\ref{subsec:lit_review} for a brief discussion of the existing literature on models devised for studying the spread and control of epidemics, based on game theory). 

Our model is firmly based on the premise of game theory, constituting a population $N = \{1, 2, \ldots, n\}$ of $n$ agents, each of whom is allowed to choose from a set $A = [0,1]$ of available \emph{actions}. Choosing action $0$ is equivalent to the agent confining themselves to their home and coming in contact with no other agent, whereas choosing action $1$ is tantamount to the agent going about their day as usual, with no restrictions imposed. An \emph{action profile} $a_N=(a_1,\ldots,a_n)$ is an element of the set $A^n$, with $a_{i}$ indicating the most recent action undertaken by agent $i$, for each $i \in N$. The agents are represented by the vertices of an undirected weighted graph, and the \emph{interaction} between agent $i$ and agent $j$, for distinct $i, j \in N$, is captured by the weight $g_{i,j} \in [0,1]$ of the edge connecting the vertices $i$ and $j$. We further endow agent $i$, for each $i \in N$, with an \emph{immunity power} $\tau(i) \in (0,1)$. We consider a discrete-time stochastic process indexed by $\mathbb{N}_{0}$, the set of all non-negative integers. At the beginning of the $t$-th \emph{epoch} of time, for each $t \in \mathbb{N}_{0}$, an agent $\undertilde{v}_{t}$ is chosen uniformly randomly out of $N$ and permitted to update their action. The chosen agent decides upon their action by taking stock of the \emph{state} the process is in at the beginning of that epoch, and their own \emph{utility function}, both of which are formally defined in \S\ref{sec:model}. We mention here that the state $S_{t}$ of the process, at the start of epoch $t \in \mathbb{N}_{0}$, is made up of two crucial components:
\begin{enumerate*}
\item the set $I(S_{t})$ comprising all the agents who have been infected up to and including epoch $t-1$,
\item and the action profile $a_{N}(S_{t})$ of the agents at the beginning of epoch $t$. 
\end{enumerate*}

The process mentioned above shall, henceforth, be referred to as the \emph{stochastic virus spread  process} (SVSP). In addition, we shall consider, for some of our preliminary investigations of the SVSP, a \emph{deterministic virus spread process} (DVSP) (see \S\ref{sec:model} for a more formal definition) in which the sequence $\undertilde{v} = (\undertilde{v}_{t}: t \in \mathbb{N}_{0})$ of agents is specified fully (i.e.\ the agent $\undertilde{v}_{t}$ chosen to update their action at the start of epoch $t$, for each $t \in \mathbb{N}_{0}$, is predetermined, and \emph{not} random). 

The principal questions we aim to answer in this work are those concerning the limiting distribution of the infected set $I(S_{t})$ and the limiting distribution of the action profile $a_{N}(S_{t})$ of all agents concerned, as $t \rightarrow \infty$, provided such limits exist. Such questions are pertinent not just theoretically, but also from a very practical perspective in that, in any country, the departments under the federal government that are tasked with overseeing the provision of healthcare for the population must be able to reliably predict the approximate proportion of citizens to get infected in the long run (i.e.\ when the epidemic has continued for a considerably long duration). This is necessary because such knowledge can aid in the decision of how much resources (medicines and medical equipment, hospital beds etc.) to set aside for the treatment of infected patients in the long run. The investigation of the limiting behaviour of the action profile $a_{N}(S_{t})$ as $t \rightarrow \infty$ goes on to reveal how, when such a limit exists, individuals in a population typically tend to behave once the epidemic has prevailed for a sufficiently long time. 


\subsection{A brief review of pertinent literature}\label{subsec:lit_review}
The classical \emph{compartmental models} of epidemiology (see \cite{brauer2008compartmental} for a comprehensive survey) date as far back as the early 1900s (see \cite{ross1917application}). Some of the most notable ones out of these are \emph{Susceptible-Infectious-Removed} (SIR) model (see \cite{kermack1927contribution}), the \emph{Susceptible-Infectious-Susceptible} (SIS) model (see \cite{hethcote1973asymptotic}) and the \emph{Susceptible-Exposed-Infectious-Removed} (SEIR) model (see \cite{aron1984seasonality}). In the recent years, \emph{network models} have become more popular, with the vertices or \emph{nodes} of a network representing the individuals of a population under consideration, and the edge between any two distinct nodes denoting the relationship or interaction between the two individuals those nodes represent, in such models (for instance, see \cite{newman2002spread}, \cite{pastor2015epidemic}, \cite{craig2020improving}, \cite{keeling2005networks}, \cite{lang2018analytic}, \cite{maheshwari2020network}, \cite{small2010complex}, \cite{sah2021revealing}, \cite{azizi2020epidemics}, \cite{cui2021network}, \cite{browne2021infection} etc.).

We now begin a discussion of research articles that are closely aligned in flavour with our work in this paper. We begin with \cite{aurell2022finite}, which investigates a game for a continuum of non-identical players evolving on a finite state space, with their heterogeneous interactions with other players represented via a \emph{graphon} (viewed as the limit of a dense random graph). A player's transition rates between the states depend on their control and the strength of their interaction with other players. Sufficient conditions for the existence of Nash equilibria are studied in \cite{aurell2022finite}, and the existence of solutions to a continuum of fully coupled forward-backward ordinary differential equations characterizing the Nash equilibria is proved. In \cite{vizuete2020graphon}, spectral properties of graphons are used to study stability and sensitivity to noise of deterministic SIS epidemics over large networks. In particular, the presence of additive noise in a linearized SIS model is considered and a noise index is derived to quantify the deviation from the disease-free state due to noise. 

In the next couple of paragraphs, we focus on citing a few of the articles out of the vast literature that concerns itself with applying the theory of mean field games to the study of the spread of an epidemic throughout a population. In \cite{aurell2022optimal}, motivated by models of epidemic control in large populations, a \emph{Stackelberg mean field game} model between a principal and a mean field of agents evolving on a finite state space is considered, with the agents playing a non-cooperative game in which they can control their transition rates between states to minimize individual costs. An application is then proposed to an epidemic model of the SIR type in which the agents control their interaction rate and the principal is a regulator acting with non pharmaceutical interventions. In \cite{lee2021controlling}, a mean-field game model in controlling the propagation of epidemics on a spatial domain is introduced, with the control variable being the spatial velocity (introduced at first for the classical disease models, such as SIR), and fast numerical algorithms based on proximal primal-dual methods are provided. In \cite{lee2022mean}, a mean-field variational problem in a spatial domain, controlling the propagation of a pandemic by the optimal transportation strategy of vaccine distribution, is investigated. In \cite{olmez2022modeling}, an agent's decision as to whether to be socially active in the midst of an epidemic is modeled as a mean-field game with health-related costs and activity-related rewards. By considering the fully and partially observed versions of this problem, the role of information in guiding an agent's rational decision is highlighted. In \cite{olmez2022does}, how the evolution of an infectious disease in a large heterogeneous population is governed by the self-interested decisions (to be socially active) of individual agents is studied based on a mean-field type optimal control model. The model is used to investigate the role of partial information on an agent's decision-making, and study the impact of such decisions by a large number of agents on the spread of the virus in the population. 

In \cite{cho2020mean}, a mean-field game model is proposed in which each of the agents chooses a dynamic strategy of making contacts, given the trade-off of gaining utility but also risking infection from additional contacts. Both the \emph{mean-field equilibrium strategy}, which assumes that each agent acts selfishly to maximize their own utility, and the \emph{socially optimal strategy}, which maximizes the total utility of the population, are computed and compared with each other. An additional cost is also included as an incentive to the agents to change their strategies, when computing the socially optimal strategies. The \emph{price of anarchy} of this system is computed to understand the conditions under which large discrepancies between the mean-field equilibrium strategies and the socially optimal strategies arise, which is when intervening public policy would be most effective. In \cite{doncel2022mean}, a mean field game model of SIR dynamics is proposed in which players choose when to get vaccinated. It is shown that this game admits a unique mean-field equilibrium that consists of vaccinating aggressively at a maximal rate for a certain amount of time and then not vaccinating, and it is shown that this equilibrium has the same structure as the vaccination strategy that minimizes the total cost. A very similar problem is studied in \cite{gaujal2021vaccination} that focuses on a virus propagation dynamics in a large population of agents, with each agent being in one of three possible states (namely, susceptible, infected and recovered) and with each agent allowed to choose when to get vaccinated. It is shown that this system admits a unique symmetric equilibrium when the number of agents goes to infinity, and that the vaccination strategy that minimizes the social cost has the same threshold structure as the mean field equilibrium, though the latter has a shorter threshold. In \cite{hubert2018nash}, the newborn, non-compulsory vaccination in an SIR model with vital dynamics is studied, with the evolution of each individual modeled as a Markov chain and their decision to vaccinate aimed at optimizing a criterion depending on the time-dependent aggregate (societal) vaccination rate and the future epidemic dynamics. The existence of a Nash mean field game equilibrium among all individuals in the population is established. In \cite{laguzet2015individual}, techniques from the mean field game theory are used to examine whether, in an SIR model, egocentric individuals (i.e.\ whose actions are driven by self-interest when it comes to getting vaccinated) can reach an equilibrium with the rest of the society, and it is shown that an equilibrium exists. The individual best vaccination strategy (with as well as without discounting) is completely characterized, a comparison is made with a strategy based only on overall societal optimization, and a situation with a non-negative price of anarchy is exhibited. In \cite{laguzet2016equilibrium}, individual optimal vaccination strategies in an SIR model are analyzed. It is assumed that the individuals vaccinate according to a criterion taking into account the risk of infection, the possible side effects of the vaccine and the overall epidemic course, that the vaccination capacity is limited, and that the individual discounts the future at a given positive rate. Under these assumptions, an equilibrium between the individual decisions and the epidemic evolution is shown to exist. In \cite{salvarani2018optimal}, a model of agent-based vaccination campaign against influenza with imperfect vaccine efficacy and durability of protection is considered. The existence of a Nash equilibrium is proved and a novel numerical method is proposed to find said equilibrium. Various aspects of the model are also discussed, such as the dependence of the optimal policy on the imperfections of the vaccine, the best vaccination timing etc. 

In \cite{hubert2022incentives}, a general mathematical formalism is introduced to study the optimal control of an epidemic via incentives to lockdown and testing, and the interplay between the government and the population, while an epidemic is spreading according to the dynamics given by a stochastic SIS model or a stochastic SIR model, is modeled as a principal–agent problem with moral hazard. Although, to limit the spread of the virus, individuals within a given population can choose to reduce interactions among themselves, this cannot be perfectly monitored by the government and it comes with certain social and monetary costs for the population. One way to mitigate such costs and encourage social distancing, lockdown etc.,\ is to put in place an incentive policy in the form of a tax or subsidy.  In addition, the government may also implement a testing policy in order to know more precisely the spread of the epidemic within the country, and to isolate infected individuals. It is verified via numerical results that if a tax policy is implemented, the individuals in the population are encouraged to significantly reduce interactions among themselves, and if the government also adjusts its testing policy, less effort is required on the part of the population to enforce social distancing, lockdown upon itself, and the epidemic is largely contained by the targeted isolation of positively-tested individuals. In \cite{cooney2022social}, a model for the evolution of sociality strategies in the presence of both a beneficial and costly contagion is investigated, and a social dilemma is identified in that the evolutionarily-stable sociality strategy is distinct from the collective optimum (i.e.\ the level of sociality that would be best for all individuals) -- in particular, the level of social interaction in the former is greater (respectively less) than the social optimum when the good contagion spreads more (respectively less) readily than the bad contagion. Finally, we cite \cite{roy2022recent}, which provides a state-of-the-art update on recent advances in the mean field approach that can be used very effectively in analyzing a dynamical modeling framework, known as a continuous time Markov decision process, for epidemic modeling and control. 

 \subsection{Organization of the paper}\label{subsec:org} The model that we investigate in this paper, along with all the pertinent definitions, has been described formally in \S\ref{sec:model}, although we did allude to it briefly in \S\ref{sec:intro}. \S\ref{sec:model} also includes some observations and lemmas concerning the the deterministic virus spread process (also mentioned previously in \S\ref{sec:intro}). The main results of this paper, namely Theorems~\ref{theo_1}, \ref{thm_ac_1}, \ref{theo_1.5}, \ref{thm_ac_1.5}, \ref{theo_2}, \ref{thm_ac_3}, \ref{theo_3}, \ref{thm_ac_4}, \ref{theo_4} and \ref{thm_ac_5} are stated in \S\ref{sec:results}, along with relevant discussions regarding the conclusions drawn from them. Simulations exploring the cardinality of the infected set for the first several epochs of the process, thereby yielding good approximations to the limit that it converges to, are given in \S\ref{sec:simulation}. A summary of what we have been able to achieve in this paper, along with directions of research on this as well as related topics that we wish to pursue in the future, is provided in \S\ref{sec:discussion}. The proofs of the main results (that have been stated in  \S\ref{sec:results}) are deferred to \S\ref{sec:appendix}, the Appendix.
 


\section{Formal description of the model}\label{sec:model}
\noindent Recall, from the second paragraph of \S\ref{subsec:overview}, the brief introduction to the model we consider in this paper. Here, we formalize the model by providing mathematical definitions to the crucial quantities involved in it.

The process described in \S\ref{subsec:overview} is said to be in the \emph{state} $S = (I,a_{N})$ if $I \subseteq N$ denotes the set of infected agents and $a_N$ denotes the action profile at that time. Given a state $S$, we denote by $I(S)$ the corresponding set of infected agents, and by $a_{\widehat{N}}(S) = (a_{i}(S): i \in \widehat{N})$ the tuple in which $a_{i}(S)$ represents the action of the $i$-th agent for all $i \in \widehat{N}$, for any subset $\widehat{N}$ of $N$. In particular, if $\widehat{N} = N \setminus \{j\}$, we abbreviate the notation $a_{N \setminus \{j\}}(S)$ by $a_{-j}(S)$, and for any $a \in A$, we denote by $(a \vee a_{-j}(S))$ the tuple $(a_{1}(S), \ldots, a_{i-1}(S), a, a_{i+1}(S), \ldots, a_{n}(S))$. We denote by $\mathcal{S}$ the set of all possible states. 

The viral exposure $r_i(S)$ that agent $i$ is subjected to, when the process is in state $S$, is defined as
\[
 r_i(S) = 
  \begin{cases} 
   \left(\frac{\sum_{j\in I\setminus \{i\}}g_{ij}a_j(S)}{\sum_{j\in N\setminus \{i\}}g_{ij}a_j(S)}\right) & \text{if } \sum_{j\in N\setminus \{i\}}g_{ij}a_j(S) \neq 0, \\
   0 & \text{if } \sum_{j\in N\setminus \{i\}}g_{ij}a_j(S) = 0.
  \end{cases}
\]
Such a definition can be motivated is as follows: the sum $\sum_{j\in N\setminus \{i\}}g_{ij}a_j(S)$ equals $0$ if and only if, for every $j \in N \setminus \{i\}$, either $g_{i,j} = 0$, thereby indicating that there is no interaction whatsoever between agent $i$ and agent $j$, or else the action $a_{j}(S) = 0$, which indicates that agent $j$ has chosen to socially distance themselves from all other agents -- in either of these two scenarios, agent $i$ runs no risk of being infected by agent $j$. Therefore, the viral exposure agent $i$ is subjected to is $0$ whenever $\sum_{j\in N\setminus \{i\}}g_{ij}a_j(S) = 0$. 

The utility function of agent $i$, when the process is in state $S$, is defined as
\begin{equation}\label{ut}
 u_i(S) = 
  \begin{cases} 
   1+f(a_i(S)) & \text{if } i\notin I(S) \text{ and }  a_i(S) r_i(S)\leqslant \tau(i), \\
   f(a_i(S)) & \text{if either } i\in I(S) \text{ or } a_i(S) r_i(S)> \tau(i)
  \end{cases}
\end{equation}
where $f:[0,1] \to [0,1]$ is a strictly increasing function. Intuitively, if agent $i$ is neither already infected during the current epoch (which is indicated by the condition $i \notin I(S)$) nor runs the risk of being infected in the next epoch (which is indicated by the condition $a_i(S) r_i(S)\leqslant \tau(i)$, i.e.\ their action multiplied by the viral exposure they have been subjected to does not exceed their immunity power), they enjoy a `reward' of amount 1 in addition to the utility $f(a_{i}(S))$ that they receive because of their chosen action (note that the strictly increasing nature of $f$ ensures that the more they go out in society, the more utility they get). Else, they are deprived of such a reward and must settle for the utility value $f(a_{i}(S))$. 

We now formally describe how agent $i$ responds if they are chosen to update their action at the beginning of an epoch when the system is in state $S$. We call this the \emph{best response} by agent $i$ at state $S$, denoted $b_i(S)$, and it is defined as
\begin{equation}
b_i(S)=\argmax_{a \in [0,1]} u_i\left(I(S),(a \vee a_{-i}(S))\right).\label{best_response}
\end{equation}
In words, this is the set of actions $a$ by agent $i$ that allow them to maximize their utility function (note that the utility function, as defined in \eqref{ut}, is a function of the state, and the state here constitutes $I(S)$ as the infected set and $(a \vee a_{-i}(S))$ as the action profile). 

We now summarize the stochastic process we focus on in this paper. We denote by $S_{t} = (I(S_{t}), a_{N}(S_{t}))$ the state of the process at the start of epoch $t$, for $t \in \mathbb{N}_{0}$. Let $\undertilde{v}_{t}$ denote the agent chosen uniformly randomly out of $N$ at the start of epoch $t$, and we define the infinite sequence $\undertilde{v} = (\undertilde{v}_{t}: t \in \mathbb{N}_{0}) \in N^{\mathbb{N}_0}$. Given $\undertilde{v}_{t} = i$, agent $i$ updates her action by playing the best response $b_{i}(S_{t})$ as defined in \eqref{best_response}. We are now at an \emph{intermediate} state $\hat{S}_{t} = \left(I(S_{t}), (b_{i}(S_{t}) \vee a_{-i}(S_{t}))\right)$ that may be imagined to occur at the midpoint of epoch $t$. We now update the infected set to $I(S_{t+1}) = I(S_{t}) \cup \{j: a_{j}(\hat{S}_{t})r_{j}(\hat{S}_{t}) > \tau(j)\}$, and finally, at the end of epoch $t$ and the beginning of epoch $t+1$, the process is in state $S_{t+1} = \left(I(S_{t+1}), (b_{i}(S_{t}) \vee a_{-i}(S_{t}))\right)$ (which tells us that $a_{N}(S_{t+1}) = (b_{i}(S_{t}) \vee a_{-i}(S_{t}))$).

The following lemma summarizes the best response of an agent at a state depending on whether or not they are infected at that state. 
\begin{lemma}\label{rem_0}
	Let $i \in N$ be an agent and $S \in \mathcal{S}$ be a state. Then, 
	 \[
	 b_i(S) = 
  \begin{cases} 
   1 & \text{if } i\in I(S), \\
   1 & \text{if } i\notin I(S) \text{ and } r_i(S)=0, \\
   \min \left\{1,\frac{\tau(i)}{r_i(S)}\right\} & \text{if } i\notin I(S) \text{ and } r_i(S)\neq 0.
  \end{cases} 
	 \]
	 
	
	
\end{lemma}

\noindent The proof of this lemma can be found in Appendix \ref{sec:appendix_A}. There are two key messages to take away from Lemma \ref{rem_0}. The first of these is that the best response of an agent $i$ at any state $S$ is unique, which is why we have, henceforth, presented $b_i(S)$ as an element of $A$ (which is more convenient than writing it as a singleton subset of $A$). The second is that once an agent is infected or runs no risk of becoming infected (i.e.\ the viral exposure is $0$), they choose to go out with no restrictions imposed on their movements.

%
		

Although we alluded to it in \S\ref{sec:intro}, we recall here the definition of the deterministic virus spread process (DVSP). Given a (deterministic) agent sequence $\undertilde{v}$, the DVSP $S = (S_{t}: t \in \mathbb{N}_{0})$ induced by $\undertilde{v} = (\uti_{t}: t \in \mathbb{N}_{0})$, with $S_{t}$ indicating the state of the process just before epoch $t$ commences, is defined in a manner identical to the stochastic virus spread process (SVSP) described above, with the only difference being that, instead of choosing an agent randomly at the start of each epoch, the agent $\uti_t$ is chosen at the start of epoch $t$ to update their action, for each $t \in \mathbb{N}_{0}$.
Whenever the agent sequence $\uti$ is not clear from the context, we shall denote the DVSP $S$ (induced by $\undertilde{v}$) by $S(\undertilde{v})$ to emphasize its dependence on $\uti$. In this case, $S_{t}(\undertilde{v})$ will denote the state of the process at the start of epoch $t$. 

In what follows, we make a few observations about the DVSP $S(\uti)$ that we shall use frequently throughout the paper.  
\begin{observation}\label{rem_0.25}
	Let $S$ be the DVSP induced by $\undertilde{v}$. Then, for all $t\in \mathbb{N}_0$, 
	\begin{enumerate}[(i)]
		\item $I(S_t)=I(\hat{S}_t)$ and $a_{N}(\hat{S}_t)=a_{N}(S_{t+1})$, 
		\item if $b_{\uti_t}(S_t)=a_{\uti_t}(S_t)$, then $S_t=\hat{S}_t=S_{t+1}$, 
		\item if $I(S_t)=I(S_{t+1})$, then $\hat{S}_t=S_{t+1}$.
	\end{enumerate}
\end{observation}
\noindent Note here that $\hat{S}_t$ indicates the intermediate state of the process at the midpoint of epoch $t$, for each $t \in \mathbb{N}_{0}$.

\begin{observation}\label{rem_0.5}
	For any fixed $i\in N$, if  $\uti_t\neq i$ for some $t \in \mathbb{N}_0$, then  $a_i(S_t)=a_i(\hat{S}_t)=a_i(S_{t+1})$. By repeated applications of this observation, we are able to conclude the following: if $\uti_t\neq i$ for all $t \in [t',t'']$ with $t'<t''$, then this yields $a_i(S_{t})=a_i(S_{t'})$ for all $t\in [t',t'']$. 
\end{observation}

\begin{observation}	\label{obs_3}
	Since the best response of an infected agent is always $1$ (see Lemma~\ref{rem_0}), $i \in I(S(t))$ and $\uti_t=i$ together imply that $a_{i}(S_{t'})=1$ for all $t'>t$.
\end{observation}

\begin{observation}\label{obs_4}
		Since the best response of an uninfected agent $i$ is 
				 \[b_i(S)=
			\begin{cases} 
  
   1 & \text{if } r_i(S)=0, \\
   \min \left\{1,\frac{\tau(i)}{r_i(S)}\right\} & \text{if }  r_i(S)\neq 0,
  \end{cases} 
\]
			
				
			by Lemma~\ref{rem_0}, hence $\uti_t=i$ and $i\notin I(S_t)$ together imply that $i\notin I(S_{t+1})$ as well.
\end{observation}

\noindent Recall that our main goal in this paper is to explore the limiting behaviours of both the cardinality of the infected set of agents and the action profile of all the agents in our population. We now show that such limits are well-defined, at the very least, for a deterministic sequence of agents: 
\begin{lemma}\label{lem_1}
The DVSP $S(\uti)$ converges for each agent sequence $\uti \in N^{\mathbb{N}_0}$. In other words, both $\lim_{t \to \infty} I(S_t(\uti))$ and $\lim_{t \to \infty} a_{N}(S_t(\uti))$ exist.
\end{lemma}
\noindent The proof of this lemma can be found in Appendix \ref{sec:appendix_A}. In view of Lemma \ref{lem_1}, we set $S_\infty(\uti) = \lim_{t \rightarrow \infty}S_{t}(\uti)$. 

The set ${N}^{\mathbb{N}_0}$ is the set  of all agent-sequences indexed by $\mathbb{N}_0$. We consider the probability space $(N^{\mathbb{N}_0}, \mathcal{F}, \mathbb{P})$ where $\mathcal{F}$ is the sigma-field generated by the cylindrical sets of ${N}^{\mathbb{N}_0}$ and $\mathbb{P}$ is the uniform probability distribution.

\begin{remark}\label{rem_3}
  \sloppy Let $N_{\infty}$ be the subset of ${N}^{\mathbb{N}_0}$ consisting of the agent-sequences where each agent moves an infinite number of times. In other words, ${N}_{\infty}= \{\uti \in {N}^{\mathbb{N}_0}: \uti_t=i \text{ for infinitely many } t, \text{ for all } i \in N\}$. It is straightforward to see that the set $N_\infty$ has probability 1 under $\mathbb{P}$, since the probability of the set ${N}^{\mathbb{N}_0} \setminus N_\infty$ is $0$.
\end{remark}

In view of Remark~\ref{rem_3}, for the rest of the paper, we shall work with the probability space $(N_\infty,\mathcal{F},\mathbb{P})$. Recall that in the stochastic virus spread  process (SVSP), before each epoch commences, an agent is chosen randomly, following the discrete uniform distribution on the set $N$, and they are allowed to update their action by playing their best response (see \eqref{best_response}) to the current state. Consequently, the SVSP is a random variable $S$ supported on the probability space $(N_\infty, \mathcal{F}, \mathbb{P})$.

For an agent $i \in N$, the random variable $t_i$ is defined as follows with respect to $(N_\infty, \mathcal{F}, \mathbb{P})$: for $\undertilde{v}\in N_\infty$, we set $t_i(\uti)= l$ if $l \in \mathbb{N}_0$ is such that $\uti_l = i$ and $\uti_k \neq i$ for all $k<l$. Note that for any $\uti\in N_\infty$, $i\in N$, and $t\in \mathbb{N}_0$ with $t\leqslant t_i(\uti)$, we have $a_i(S_0)=a_i(S_{t}(\uti))$. Let $N_1$ be the measurable function on $(N_\infty, \mathcal{F}, \mathbb{P})$ that describes the (random) set of agents who had been chosen before agent 1 was chosen for the first time, that is, $N_1(\uti)= \{i\in N \mid  t_i(\uti)<t_1(\uti)\}$. 


We now establish that $|N_1|$ follows the uniform distribution on $\{0,1,\ldots,n-1\}$. Lemma~\ref{lem_5} will be used in the proofs of the main results of this paper.
\begin{lemma}\label{lem_5}
$\mathbb{P}(|N_1|=l) = \frac{1}{n}$ for all $l \in \{0,1,\ldots,n-1\}$.
\end{lemma}
\begin{proof}
	 Since $\mathbb{P}$ is uniform and there are $n!$ possible orderings of the random times $t_1,\ldots, t_n$, each ordering of $t_1,\ldots,t_n$ has an equal probability of $\frac{1}{n!}$ to occur. We can choose $m-1$ random variables from the set $\{t_2,\ldots,t_{n}\}$ of $n-1$ random times in $\Mycomb[n-1]{m-1}$ ways. Therefore, the number of orderings that correspond to the event $| \{i \in N \mid t_i < t_1\}| = m-1$  is $\Mycomb[n-1]{m-1} \times (m-1)! \times (n-m)! $, and hence, the probability of the said event  is $\frac{\Mycomb[n-1]{m-1} \times (m-1)! \times (n-m)! }{n!}$, which is $\frac{1}{n}$. This completes the proof of the lemma.
 \end{proof}

\section{Main Results }\label{sec:results} Throughout this section, we adhere to the following assumptions, and all results stated henceforth are true when these assumptions are imposed:
\begin{enumerate}
  		\item We assume that the process starts from the initial state $S_0$ in which $a_i(S_0)=a$ for all $i\in N$, and that the initial infected set is a singleton. In particular, we set $I(S_0)=\{1\}$,
		\item We assume that each agent in $N$ is endowed with the same immunity power, i.e.\ $\tau(i)=\tau$ for all $i \in N$.
	    \item Finally, we assume that the interaction between every pair of agents is the same. In particular, we set $g_{ij}=1$ for all $(i,j)\in N^2$ with $i \neq j$.
\end{enumerate}
\begin{remark}
We emphasize here that all results in \S\ref{sec:results} hold if we assume $g_{ij}=c$ for all $(i,j) \in N^2$ with $i\neq j$, for some $c \in \mathbb{R}$.
\end{remark}

We adopt the following notations to state our results. For $a\in \mathbb{R}$, we let $\lceil a\rceil=\min\{k \in \mathbb{Z}: a\leqslant k\}$ and $\lfloor a\rfloor=\max\{k \in \mathbb{Z}: a\geqslant k\}$. For $a,b\in \mathbb{N}_0$, $[a,b]$ denotes the set $\{a,a+1,\ldots,b\}$ if $a\leqslant b$, and denotes the null set if $a>b$. For $m\geqslant \lfloor \tau(n-1)\rfloor +1$, we define the set $A_{m}$ to be the set of all ordered tuples $\vec{x}=(x_1,\ldots,x_n)$ that satisfy the following properties:
\begin{enumerate}
    \item $x_1=1$,
    \item there are precisely $m-1$ coordinates $i \in \{2, \ldots, n\}$ such that $x_{i} = 1$,
    \item each of the remaining coordinates equals $\frac{\tau m}{(1+\tau)m-\tau(n-1)}$.\footnote{As $m>(n-1)\tau$, $\frac{\tau m}{(1+\tau)m-\tau(n-1)}$ is strictly less than 1.}
\end{enumerate}
 A couple of facts follow immediately from the above definition. The first is that $A_{n}$ is the singleton set $\{\vec{1}\}$, where $\vec{1}$ is the $n$-dimensional tuple in which each coordinate equals $1$. The second is that $|A_{m}| = \Mycomb[n-1]{m-1}$ for each $m$ for which $A_{m}$ is well-defined, since we need only choose the $m-1$ coordinates out of $2, \ldots, n$ that equal $1$.


\subsection{Results when $a=0$}\label{subsec:a=0}
Here, we consider the situation where the (common) initial action $a$ equals $0$. Theorem \ref{theo_1} provides the limiting distribution of the infected set for arbitrary values of $\tau$. Let  
\begin{equation}\label{alpha}
\alpha=\min\left\{\left \lceil\frac{1}{\tau}\right \rceil,n\right\}.
\end{equation}

\begin{theorem}\label{theo_1} Suppose $a=0$. Then the limiting distribution of the infected set is as follows, with $\alpha$ as defined in \eqref{alpha}:

	\[
	\mathbb{P}(I(S_\infty)=J)= 
	\begin{cases} 
		1-\frac{\alpha-1}{n} & \text{if } J=\{1\} , \\\\
		\frac{1}{n \times \Mycomb[n-1]{m-1}} & \text{if } 1\in J \text{ and } |J|=m \text{ where } m\in[2,\alpha], 
		\\\\ 
		0 & \text{ otherwise }.
	\end{cases}
	\]





	
\end{theorem}
The proof of the theorem can be found in Appendix \ref{sec:appendix_C}. Next, we proceed to explore the limiting distribution of the action profile, and this is found to be dependent on the value of $\tau$. Accordingly, the statement of Theorem~\ref{thm_ac_1} is split into two parts on the basis of whether $\tau$ exceeds $(n-1)^{-1}$ or not. We introduce the quantity
\begin{equation}\label{beta}
\beta=\min \{\lfloor(n-1)\tau \rfloor+1,\alpha+1\}.
\end{equation}

\begin{theorem}\label{thm_ac_1}
 Suppose $a=0$. For $\tau\geqslant \frac{1}{n-1}$, the limiting distribution of the action profile may be described as follows: 
\[
 \mathbb{P}(a_{N}(S_\infty)=\vec{x})= 
  \begin{cases} 
   1-\frac{\alpha-\beta+1}{n} & \text{if } \vec{x}=\vec{1}  , \\ \\
   \frac{1}{n \times \Mycomb[n-1]{m-1}} & \text{if } 
   \vec{x}\in A_m \text{ for some } m\in [\beta,\alpha],\\ \\
   0 & \text{ otherwise, } 
  \end{cases}
\]





	whereas for $\tau< \frac{1}{n-1}$, we have 
\[
 \mathbb{P}(a_{N}(S_\infty)=\vec{x})= 
  \begin{cases} 
   \frac{1}{n \times \Mycomb[n-1]{m-1}} & \text{if } 
   \vec{x}\in A_m \text{ for some } m\in [1,n],\\ \\
   0 & \text{ otherwise, } 
  \end{cases}
\]
with $\alpha$ and $\beta$ as defined in \eqref{alpha} and \eqref{beta} respectively.
\end{theorem}
\noindent The proof of this theorem can be found in Appendix \ref{sec:appendix_E}.


A brief discussion is in order regarding some of the startling findings that may be deduced from the two theorems of \S\ref{subsec:a=0}. Theorem~\ref{theo_1} reveals that, if we consider \emph{any} two \emph{different} subsets $J_{1}$ and $J_{2}$ of $N$, the probabilities $\mathbb{P}(I(S_{\infty}) = J_{1})$ and $\mathbb{P}(I(S_{\infty}) = J_{2})$ are the same as long as $J_{1}$ and $J_{2}$ have the same cardinality and either both of them contain agent $1$ or neither contains agent $1$. We note that the number of subsets $J$ of $N$ with $1 \in J$ and $|J| = m$ is given by $\Mycomb[n-1]{m-1}$, so that summing $\mathbb{P}(I(S_{\infty}) = J)$ over all such $J$ yields $\mathbb{P}\left(\left|I(S_{\infty})\right| = m\right) = n^{-1}$ for \emph{each} $m \in [2,\alpha]$. These observations suggest a rather close resemblance that the limiting distribution of the infected set, as well as the limiting distribution of its cardinality, bears with suitably defined discrete uniform distributions. In fact, for $\tau \leqslant n^{-1}$, we have $\alpha = n$, reducing the distribution of $|I(S_{\infty})|$ to precisely the discrete uniform distribution on $\{ 1,2,\ldots, n \}$. This uniform structure is somewhat marred when $\tau > n^{-1}$. For example, when $n=5$, $a=0$ and $\tau= 0.25$, we have
\begin{align}
 &\mathbb{P}(|I(S_{\infty})|=1)=\frac{2}{5},\ \mathbb{P}(|I(S_{\infty})|=2)=\frac{1}{5},\ \mathbb{P}(|I(S_{\infty})|=3)=\frac{1}{5},\nonumber\\& \mathbb{P}(|I(S_{\infty})|=4)=\frac{1}{5},\ \mathbb{P}(|I(S_{\infty})|=5)=0,   
\end{align}
whereas if $\tau=0.4$, the probability distribution changes to 
\begin{align}
&\mathbb{P}(|I(S_{\infty})|=1)=\frac{3}{5},\ \mathbb{P}(|I(S_{\infty})|=2)=\frac{1}{5},\ \mathbb{P}(|I(S_{\infty})|=3)=\frac{1}{5},\nonumber\\& \mathbb{P}(|I(S_{\infty})|=4)=0,\ \mathbb{P}(|I(S_{\infty})|=5)=0.    
\end{align}
An intuitive explanation for this phenomenon is that with higher immunity, i.e.\ higher value of $\tau$, the disease is less likely to spread to the entire community, instead having a higher probability of remaining confined to the initial infected set.

Conclusions of a similar flavour can be drawn as a consequence of Theorem~\ref{thm_ac_1}. For \emph{any} two \emph{different} ordered tuples $\vec{x}$ and $\vec{y}$ that belong to the same $A_{m}$, the probabilities $\mathbb{P}(a_{N}(S_{\infty}) = \vec{x})$ and $\mathbb{P}(a_{N}(S_{\infty}) = \vec{y})$ are equal, for both the cases $\tau \geqslant (n-1)^{-1}$ and $\tau < (n-1)^{-1}$. Moreover, since $|A_{m}| = \Mycomb[n-1]{m-1}$, we obtain $\mathbb{P}(a_{N}(S_{\infty}) \in A_{m}) = n^{-1}$ for \emph{every} $m \in [\beta,\alpha]$ when $\tau \geqslant (n-1)^{-1}$ and for every $m \in [1,n]$ when $\tau < (n-1)^{-1}$. These are, once again, reminiscent of suitably defined discrete uniform distributions. 

A connection may be established between Theorem~\ref{theo_1} and Theorem~\ref{thm_ac_1}, for the case where $\tau \geqslant (n-1)^{-1}$, via the following fact whose justification has been included in the proof of Theorem~\ref{thm_ac_1} in \S\ref{subsec:proof_theo_1} of the Appendix (\S\ref{sec:appendix_C}): for any DVSP $S(\uti)$, if the limiting infected set has cardinality $m\in [\beta,\alpha]$ (note that $\beta\geqslant 2$), the limiting action profile will be a tuple in $A_m$, with all infected agents choosing action $1$ and all uninfected agents choosing action $\tau m[(1+\tau)m-\tau(n-1)]^{-1}$. On the other hand, if the limiting infected set for the DVSP $S(\uti)$ has cardinality strictly less than $\beta$, the final action profile becomes $\vec{1}$, signifying that \emph{all} agents choose action $1$ in the long run.

	

\subsection{Results when $a=1$}\label{subsec:a=1}
Here, we consider the situation where the (common) initial action $a$ equals $1$. The following theorem provides the limiting distribution of the set of infected agents: 

\begin{theorem}\label{theo_1.5} Suppose $a=1$. If $\tau\geqslant \frac{1}{n-1}$, the limiting distribution of the infected set is given by
\begin{center}
    $\mathbb{P}(I(S_\infty)=\{1\})= 
   1$,
 \end{center} 
whereas if $\tau < \frac{1}{n-1}$, the limiting distribution is given by
\[
 \mathbb{P}(I(S_\infty)=J)= 
  \begin{cases} 
   \frac{1}{n^2} & \text{if } 1\in J \text{ and } |J|=n-1,
   \\ \\
    1-\frac{n-1}{n^2} & \text{if } |J|=n, \text{ i.e., } J=N, \\ \\
   0 & \text{ otherwise. } 
  \end{cases}
\]    

\end{theorem}

\noindent The proof of this theorem can be found in Appendix \ref{sec:appendix_C}. Note that since there are $n-1$ many sets $J$ such that $1\in J \text{ and } |J|=n-1$, the above display exhibits a valid probability distribution.

\begin{theorem}\label{thm_ac_1.5}
Suppose $a=1$. If $\tau\geqslant \frac{1}{n-1}$, the limiting distribution of the action profile is given by
 \begin{center}
    $\mathbb{P}(a_{N}(S_\infty)=\vec{1})= 
   1$,
 \end{center} 
	whereas if $\tau< \frac{1}{n-1}$, the limiting distribution becomes
		\[ 
  		\mathbb{P}(a_{N}(S_\infty)=\vec{x})=
  	\begin{cases}
   1-\frac{n-1}{n^2} &\text{ if } \vec{x}=\vec{1}, \\ \\
  \frac{1}{n^2}& \text{ if }  \vec{x}\in A_{n-1}, \\ \\
  	0 & \text{ otherwise. }
  	
  	\end{cases}
\] 
\end{theorem}

\noindent The proof of this theorem can be found in Appendix \ref{sec:appendix_E}.

We draw the reader's attention to the fact that the results of \S\ref{subsec:a=1} differ quite a bit in appearance from those in \S\ref{subsec:a=0}. While the limiting distribution of the infected set, for $a=0$, is supported on \emph{all} subsets of $N$ that contain $1$ and that have sizes bounded above by $\alpha$ (Theorem~\ref{theo_1}), the infected set, for $a = 1$, converges to the singleton $\{1\}$ when $\tau \geqslant (n-1)^{-1}$, and its limiting distribution is supported on \emph{only} those subsets of $N$ that contain $1$ and have cardinality at least $n-1$ when $\tau < (n-1)^{-1}$ (Theorem~\ref{theo_1.5}). In some sense, for $a=0$, the limiting distribution is ``spread out" over a wider support, while for $a=1$, it is more ``concentrated".

Likewise, for $a = 0$, the limiting distribution of the action profile is supported on \emph{all} $A_{m}$ with $m \in \{n\} \cup [\beta,\alpha]$ when $\tau \geqslant (n-1)^{-1}$, and it is supported on \emph{all} $A_{m}$ with $m \in \{1, \ldots, n\}$ when $\tau < (n-1)^{-1}$ (Theorem~\ref{thm_ac_1}). In contrast, for $a=1$, the action profile converges to $\vec{1}$ when $\tau \geqslant (n-1)^{-1}$, and the limiting distribution of the action profile is supported on \emph{just} $A_{n} \cup A_{n-1}$ when $\tau < (n-1)^{-1}$ (Theorem~\ref{thm_ac_1.5}).

\subsection{Results when $0<a\leqslant \tau$ and $a\neq 1$}\label{sec:a<t}

In this subsection, we consider the case where the (common) initial action $a$ lies strictly between $0$ and $1$, and is bounded above by $\tau$. Let 
\begin{equation}\label{halpha}
\halpha= \max\left\{1,\left\lceil \frac{\frac{1}{\tau}-(n-1)a}{1-a}\right\rceil\right\}.
\end{equation}

\begin{theorem}\label{theo_2}
Suppose  $0<a\leqslant \tau$ and $a\neq 1$. Further, suppose $\tau\geqslant \frac{1}{n-1}$. Then  the limiting distribution of the infected set is given by	 
 	\[
	\mathbb{P}(I(S_\infty)=J)= 
	\begin{cases} 
		1-\frac{\hat{\alpha}-1}{n} & \text{if } J=\{1\}  , \\ \\
		\frac{1}{n \times \Mycomb[n-1]{m-1}} & \text{if }  1\in J \text{ and } |J|=m \text{ where } m\in[2, \halpha],
		\\ \\
		0 & \text{ otherwise.}
	\end{cases}
	\]

 


	\end{theorem}

\noindent The proof of this theorem can be found in Appendix \ref{sec:appendix_C}. Next, we proceed to describe the limiting distribution of the action profile. We introduce the following notation in order to state our next result:
\begin{equation}\label{hbeta}
\hbeta=\min \{\lfloor(n-1)\tau \rfloor+1,\halpha+1\}.
\end{equation}

\begin{theorem}\label{thm_ac_3}
Suppose  $0<a\leqslant \tau$ and $a\neq 1$. Further, suppose $\tau\geqslant \frac{1}{n-1}$.  Then the limiting distribution of the action profile is given by	 
		\[
 \mathbb{P}(a_{N}(S_\infty)=\vec{x})= 
  \begin{cases} 
  1-\frac{\hat{\alpha}-\hat{\beta}+1}{n} & \text{if } \vec{x}=\vec{1}  , \\ \\
   \frac{1}{n \times \Mycomb[n-1]{m-1}} & \text{if } 
   \vec{x}\in A_m \text{ for some } m\in [\hat{\beta},\hat{\alpha}],\\ \\
   0 & \text{ otherwise. } 
  \end{cases}
\]
 			

 
\end{theorem}

\noindent The proof of the theorem can be found in Appendix \ref{sec:appendix_E}.

\begin{remark}
If one sets $a=0$ in the conclusion of Theorem \ref{theo_2}, one gets back the conclusion of Theorem \ref{theo_1} for $\tau\geqslant (n-1)^{-1}$. However, Theorem \ref{theo_1} is more general in terms of its coverage of the values of $\tau$. In a similar manner, setting $a=0$ in Theorem \ref{thm_ac_3} yields Theorem \ref{thm_ac_1} for the case of $\tau\geqslant (n-1)^{-1}$.


\end{remark}

Discussions of findings of a flavour similar to those in \S\ref{subsec:a=0} can be included here as well. Even if $J_{1}$ and $J_{2}$ are two \emph{different} subsets of $N$, Theorem~\ref{theo_2} shows that the probabilities $\mathbb{P}(I(S_{\infty})=J_{1})$ and $\mathbb{P}(I(S_{\infty})=J_{2})$ are the same as long as $J_{1}$ and $J_{2}$ have the same cardinality and either both contain $1$ or neither does. Summing over all subsets of $N$ that contain $1$ and are of cardinality $m$, we obtain $\mathbb{P}\left(\left|I(S_{\infty})\right|=m\right) = n^{-1}$ for \emph{each} $2 \leqslant m \leqslant \halpha$. Likewise, for any two \emph{different} ordered tuples $\vec{x}$ and $\vec{y}$, Theorem~\ref{thm_ac_3} shows that the probabilities $\mathbb{P}(a_{N}(S_\infty)=\vec{x})$ and $\mathbb{P}(a_{N}(S_\infty)=\vec{y})$ are the same as long as both $\vec{x}$ and $\vec{y}$ belong to the same $A_{m}$. Summing over all members of an $A_{m}$ yields $\mathbb{P}(a_{N}(S_\infty)\in A_{m}) = n^{-1}$ for \emph{every} $\hbeta \leqslant m \leqslant \halpha$.

\subsection{Results when $\tau<a<1$.}\label{sec_t<a} 
In this subsection, we consider the scenario where the (common) initial action $a$ is strictly greater than $\tau$. We introduce the following notations to facilitate the stating the results that follow: 
Let 
\begin{equation}\label{talpha_balpha}
\talpha=\max\left\{1, \left\lceil \frac{1-(n-1)a\tau}{\tau(1-a)}\right\rceil\right\} \quad\mbox{ and }\quad \balpha =\left\lfloor \frac{(n-1)a\tau}{a-\tau(1-a)}\right\rfloor+1.
\end{equation}
Note that by our assumption on $\tau$, we have $\balpha \geqslant 2$.
 
\begin{theorem}\label{theo_3}
Suppose $\frac{1}{n-1}\leqslant \tau<a<1$. If $\talpha+1 \leqslant \balpha$, the limiting distribution of the infected set is given by	 	 
\[
		\mathbb{P}(I(S_\infty)=J)= 
		\begin{cases} 
			1-\frac{\tilde{\alpha}-1}{n} & \text{if } J=\{1\}  , \\ \\
			\frac{1}{n \times \Mycomb[n-1]{m-1}} & \text{if } 
			1\in J \text{ and } |J|=m \text{ where } m\in [2,\talpha], \\ \\
			0 &  \text{ otherwise, }
		\end{cases}
		\]


 
whereas if $2\leqslant \balpha< \talpha+1$, the limiting distribution is given by
	\[
	\mathbb{P}(I(S_\infty)=J)= 
	\begin{cases} 
		1-\frac{\tilde{\alpha}-1}{n} & \text{if } J=\{1\}  , \\ \\
		\frac{1}{n \times \Mycomb[n-1]{m-1}} & \text{if } 
		1\in J \text{ and } |J|=m \text{ where } m\in[2, \balpha-1] \\ \\
		\frac{\eta(\talpha,\balpha,n)}{n-1} & \text{if } 1\in J \text{ and } |J|=n-1\\ \\
		\frac{\talpha-(\balpha-1)}{n}-\eta(\talpha,\balpha,n) & \text{if } |J|=n, \text{ i.e., } J=N,\\ \\
		0 &  \text{ otherwise,}
	\end{cases}
	\]

\noindent where $\eta(\talpha,\balpha,n)=\frac{(n-1)!}{n^3}\sum_{w=\balpha-1}^{\talpha-1}\frac{1}{(n-w-2)!}\sum_{t= w+1}^{\infty}  \left(\frac{\stirling{t-1}{w}}{n^{t-1}}\right)$, and $\stirling{p}{q}$ is the Stirling number of the second kind with parameters $p$ and $q$.  
\end{theorem}

\noindent The proof of this theorem can be found in Appendix \ref{sec:appendix_C}. As in the previous subsections, we now proceed to explore the limiting distribution of the action profile.  The following notations will be helpful in presenting the results:
\begin{equation}\label{tbeta_bbeta}
\tbeta=\min \{\lfloor(n-1)\tau \rfloor+1,\talpha+1\} \mbox{ and } \bbeta=\min \{\lfloor(n-1)\tau \rfloor+1,\balpha\}.
\end{equation}
 
\begin{theorem}\label{thm_ac_4}
	Suppose $\frac{1}{n-1}\leqslant \tau<a<1$. If $\talpha+1 \leqslant \balpha$, the limiting distribution of the action profile is given by

\[
 \mathbb{P}(a_{N}(S_\infty)=\vec{x})= 
  \begin{cases} 
  1-\frac{\tilde{\alpha}-\tilde{\beta}+1}{n} & \text{if } \vec{x}=\vec{1}  , \\ \\
   \frac{1}{n \times \Mycomb[n-1]{m-1}} & \text{if } 
   \vec{x}\in A_m \text{ for some } m\in [\tilde{\beta},\tilde{\alpha}],\\ \\
   0 & \text{ otherwise, } 
  \end{cases}
\]

whereas if $2\leqslant \balpha< \talpha+1$, the limiting distribution is given by
\[
 \mathbb{P}(a_{N}(S_\infty)=\vec{x})= 
  \begin{cases} 
 1+\frac{\bbeta-\balpha}{n}-\eta(\talpha,\balpha,n) & \text{if } \vec{x}=\vec{1}  , \\ \\
   \frac{1}{n \times \Mycomb[n-1]{m-1}} & \text{if } 
   \vec{x}\in A_m \text{ for some } m\in [\bbeta,\balpha-1],\\ \\
   \frac{\eta(\talpha,\balpha,n)}{n-1}& \text{if }  \vec{x}\in A_{n-1}, \\ \\ 
   0 & \text{ otherwise,} 
  \end{cases}
\]
\noindent where $\eta(\talpha,\balpha,n)=\frac{(n-1)!}{n^3}\sum_{w=\balpha-1}^{\talpha-1}\frac{1}{(n-w-2)!}\sum_{t= w+1}^{\infty}  \left(\frac{\stirling{t-1}{w}}{n^{t-1}}\right)$, and $\stirling{p}{q}$ is the Stirling number of the second kind with parameters $p$ and $q$.  

\end{theorem}

\noindent The proof of this theorem can be found in Appendix \ref{sec:appendix_E}. Once again, observations similar to those made in \S\ref{subsec:a=0} and \S\ref{sec:a<t} can be noted here for Theorem~\ref{theo_3} and Theorem~\ref{thm_ac_4} as well, but we do not elaborate upon them to avoid the possibility of sounding repetitive. 

The next two theorems deal with situation when $\tau$ is less than or equal to $\frac{a}{(n-1)}$. Theorem \ref{theo_4} characterizes the limiting distribution of the infected set and Theorem \ref{thm_ac_5} characterizes the limiting distribution of the action profile.

\begin{theorem}\label{theo_4} Suppose $\tau<a<1$. If $\tau= \frac{a}{n-1}$, the limiting distribution of the infected set is given by
\begin{center}
    $\mathbb{P}(I(S_\infty)=N)= 
   1$,
 \end{center} 
whereas if $\tau < \frac{a}{n-1}$, the limiting distribution is given by
\[
 \mathbb{P}(I(S_\infty)=J)= 
  \begin{cases} 
   \frac{1}{n^2} & \text{if } 1\in J \text{ and } |J|=n-1,
   \\ \\
    1-\frac{n-1}{n^2} & \text{if } |J|=n, \text{ i.e., } J=N, \\ \\
   0 & \text{ otherwise. } 
  \end{cases}
\]    

\end{theorem}

\noindent The proof of this theorem can be found in Appendix \ref{sec:appendix_C}.

\begin{theorem}\label{thm_ac_5}
Suppose $\tau<a<1$. If $\tau= \frac{a}{n-1}$, the limiting distribution of the action profile is given by
 \begin{center}
    $\mathbb{P}(a_{N}(S_\infty)=\vec{1})= 
   1$,
 \end{center} 
	whereas if $\tau< \frac{a}{n-1}$, the limiting distribution becomes
		\[ 
  		\mathbb{P}(a_{N}(S_\infty)=\vec{x})=
  	\begin{cases}
   1-\frac{n-1}{n^2} &\text{ if } \vec{x}=\vec{1}, \\ \\
  \frac{1}{n^2}& \text{ if }  \vec{x}\in A_{n-1}, \\ \\
  	0 & \text{ otherwise. }
  	
  	\end{cases}
\] 
\end{theorem}

\noindent The proof of this theorem can be found in Appendix \ref{sec:appendix_E}. The results in these final set of theorems are also consistent with our intuition. It says that if $\tau$ is small enough then the eventual infected set is, either of size $n-1$ or $n$ with probability 1. 

\section{Simulation Studies}\label{sec:simulation}
\ignore{
$n=4, a=0, \tau= 0.3, \alpha= min \{\lceil 1/\tau \rceil , n \}=4$

0.3063135 0.2499990 0.2485371 0.1951504 
 
1/4 1/4 1/4 1/4
 
$n=4, a=0, \tau= 0.4, \alpha= min \{\lceil 1/\tau \rceil , n \}=3$
 0.5116577 0.2499390 0.2384033 0.0000000
 
1/2 1/4 1/4 0

Theorem 3 sudden jump in the distribution- Describing this interesting phenomenon with simulation and some intuitive justification

Question: What else should we give as simulations? 

Sayar: I think non negligible initial contagion could invite unnecessary criticism that point out our shortcoming 
Response: ok agreed

Sayar: Should we give some result where a's are all random. Why does it always end spreading everywhere? 

Response: Thanks

How is this sounding?

Simulations 
 -a. When all a' are random it is spreading
 -b. For the regions we did not cover
 -c. What happens if taus are also random but maybe not the entire region of 0 to 1 with a fixed a?
}

\noindent So far, we have provided the asymptotic distribution of the cardinality of the final infected set but it remains to understand how fast such distributions or its close approximations are reached. In this section, we provide a very thorough exploration for the cardinality of the infected set up to a few epochs of time.  Note that, this required an exact enumeration of all possible sequences as just a random sampling would fail to ensure that all such sequences are representative enough and thus yield wrong empirical distributions. We set $n=5$ and various values of $a$ and $\tau$ in Table \ref{tab:table1}.

\begin{table}[!htbp]
\centering
\resizebox{\columnwidth}{!}{
\begin{tabular}{|l|l|l|l|l|}
 \hline
$a$ & Range of $\tau$ & $\tau$ & Theoretical distribution  & Empirical distribution               \\ \hline
0     & (0,0.25)                         & 0.12                    & (0.2,0.2,0.2,0.2,0.2)     & (0.3340,0.2,0.199,0.181,0.085)        \\
      & {[}0.25,0.334)                   & 0.3                     & (0.4,0.2,0.2,0.2,0)       & (0.42,0.2,0.199,0.181,0)             \\
      & {[}0.334,0.5)                    & 0.4                     & (0.6,0.2,0.2,0,0)         & (0.601,0.2,0.199.0,0)                \\
      & {[}0.5,1)                        & 0.6                     & (0.8,0.2,0,0,0)           & (0.8,0.2,0,0,0)                       \\\hline
0.2   & (0,0.05)                         & 0.02                    & (0,0,0,0.16,0.84)         & (0,0,0,0.16,0.84)                    \\
      & 0.05                             & 0.05                    & (0,0,0,0,1)               & (0,0,0,0,1)                          \\
      & (0.05,0.25)                      &                *         & *                         &                                      \\
      & {[}0.25,0.3125]                  & 0.3                     & (0.4,0.2,0.2,0.2,0)       & (0.42,0.2,0.199,0.181,0)             \\
      & (0.3125,0.4166]                  & 0.35                    & (0.6,0.2,0.2,0,0)         & (0.601,0.2,0.199,0,0)                \\
      & (0.4166,0.6249]                  & 0.5                     & (0.8,0.2,0,0,0)           & (0.8,0.2,0.0,0)                      \\
      & (0.6249,1)                       & 0.7                     & (1,0,0,0,0)               & (1,0,0,0,0)                        \\\hline  
0.35  & (0,0.0875)                       & 0.05                    & (0,0,0,0.16,0.84)         & (0,0,0,0.16,0.84)                    \\
      & 0.0875                           & 0.0875                  & (0,0,0,0,1)               & (0,0,0,0,1)                          \\
      & (0.0875,0.25)                    & *                       & *                         & *                                    \\
      & {[}0.25-0.2592]                  & 0.255                   & (0.4,0,0,0.048,0.552)     & (0.420,0,0.001, 0.086, 0.493)        \\
      & (0.2592,0.2985]                  & 0.27                    & (0.4, 0.2,0, 0.024,0.376) & (0.420, 0.200, 0.001, 0.062, 0.317)  \\
      & (0.2985,0.3134]                  & 0.3                     & (0.6, 0.2, 0.016,0.184)   & (0.601, 0.200, 0.001, 0.017, 0.182)  \\
      & (0.3134,0.3704)                  & 0.36                    & (0.5,0.2,0.2,0,0)         & (0.601,0.2,0.199,0,0)                \\
      & {[}0.3704.0.4878]                & 0.4                     & (0.8,0.2,0,0,0)           & (0.8,0.2,0,0,0)                      \\
      & (0.4878,1]                       & 0.5                     & (1,0,0,0,0)               & (1,0,0,0,0)                         \\ \hline 
0.45  & (0,0.1125)                       & 0.1                     & (0,0,0,0.16,0.84)         & (0,0,0,0.16,0.84)                    \\
      & 0.1125                           & 0.1125                  & (0,0,0,0,1)               & (0,0,0,0,1)                          \\
      & (0.1125,0.25)                    & *                       & *                         & *                                    \\
      & {[}0.25,0.2898]                  & 0.27                    & (0.4,0,0,0.048,0.552)     & (0.420, 0.000, 0.001, 0.086, 0.493)  \\
      & (0.2898,0.3103]                  & 0.3                     & (0.6,0,0, 0.04,0.360)     & (0.601, 0.000, 0.001, 0.041, 0.358)  \\
      & (0.3103,0.3448]                  & 0.32                    & (0.6,0.2,0, 0.016,0.184)  & (0.601, 0.200, 0.001, 0.017, 0.182)  \\
      & (0.3448,0.4225]                  & 0.4                     & (0.8,0.2,0,0,0)           & (0.8,0.2,0,0,0)                      \\
      & (0.4255,1)                       & 0.5                     & (1,0,0,0,0)               & (1,0,0,0,0)                        \\ \hline  
0.6   & (0,0.15)                         & 0.1                     & (0,0,0,0.16,0.84)         & (0,0,0,0.16,0.84)                    \\
      & 0.15                             & 0.15                    & (0,0,0,0,1)               & (0,0,0,0,1)                          \\
      & (0.15,0.25)                      & *                       & *                         & *                                    \\
      & {[}0.25,0.2777]                  & 0.26                    & (0.4,0,0, 0.048,0.552)    & (0.420, 0.000, 0.001, 0.086, 0.493)  \\
      & (0.2777,0.3124]                  & 0.3                     & (0.6,0,0, 0.04,0.360)     & (0.601, 0.000, 0.001, 0.041, 0.358)  \\
      & (0.3124,0.3571]                  & 0.34                    & (0.8,0,0,0.024,0.176)     & (0.800, 0.000, 0.000, 0.024, 0.176)  \\
      & (0.3571,1)                       & 0.4                     & (1,0,0,0,0)               & (1,0,0,0,0)                          \\ \hline
0.8   & (0,0.2)                          & 0.1                     & (0,0,0,0.16,0.84)         & (0,0,0,0.16,0.84)                    \\
      & 0.2                              & 0.2                     & (0,0,0,0,1)               & (0,0,0,0,1)                          \\
      & (0.2,0.25)                       & *                       & *                         & *                                    \\
      & {[}0.25,0.2631]                  & 0.26                    & (0.4,0,0, 0.024,0.576)    & (0.420, 0.000, 0.001, 0.086, 0.493)  \\
      & (0.2631,0.2777]                  & 0.27                    & (0.6,0,0, 0.04,0.360)     & (0.601, 0.000, 0.001, 0.041, 0.358)  \\
      & (0.2777,0.2941]                  & 0.28                    & (0.8,0,0,0.024,0.176)     & (0.800, 0.000, 0.000, 0.024, 0.176)  \\  \hline
1     & (0,0.25)                         & 0.12                    & (0,0,0,0.16,0.84)         & (0,0,0,0.16,0.84)                    \\
      & {[}0.25,1)                       & 0.5                     & (1,0,0,0,0)               & (1,0,0,0,0)                        \\ \hline
\end{tabular}}
\caption{Exact enumeration of empirical distribution after 10 epochs. Here $(p_1,\ldots,p_5)$ denotes $(\mathbb{P}(|I(S_{\infty})|=1),\ldots,\mathbb{P}(|I(S_{\infty})|=5)) $ and $*$ denotes that those regions are not covered by our theoretical results.}\label{tab:table1}
\end{table}

In particular, we see that after only 10 epochs for $n=5$, we are able to reach very good approximations to the final distributions. Moreover, the number of epochs to get a close approximation is much smaller. One can see that for the very first case of $a=0, \tau=0.12$ , we have pretty good convergence to the actual distribution in 10 steps. Generally speaking, for smaller $a$ and $\tau$ it takes longer to approximately reach close to the asymptotic distribution.

\section{Conclusion}\label{sec:discussion}

\noindent  In this article, we propose a graph-theoretic model to describe the spread of a contagious disease allowing for rational interventions from agents sitting at the nodes of the graph. The agents act based on a reasonable utility function and they may (or may not) get affected if their exposure increases. We obtain the asymptotic distribution of the cardinality of the infected set as well as that of the action profile. The results reveal several interesting patterns that exhibit proximity to uniformity, as well as results that are intuitively justifiable (such as if everyone's immunity ($\tau$ value) is low to begin with, then eventually the whole population gets affected). We have given an almost complete picture of how the values of $\tau$ and $a$ impact the final distribution of the infected set and the action profile. We also observe several fascinating phase transition phenomena in our results. Through exact enumeration of all possible sequences in which the agents are picked randomly, we also show that the empirical distributions obtained mimic the corresponding empirical distributions rather closely after only around $10$ epochs from the start of the process. 

However, there are a number of questions that remain to be addressed that seem to be beyond the scope of this paper. We give the readers a brief overview of the questions we intend to pursue for similar or related models in the future. So far, we have been unable to obtain, theoretically at least, the limiting distributions for the case where $a/(n-1)< \tau < 1/(n-1)$. While this is negligible for large $a$ (keeping in mind that $a \in [0,1]$, so that a large value of $a$ indicates its closeness to $1$) or for large $n$, our simulation studies show that there are possibly only two different distributions that could lie in this space. Second, we want to relax the restriction that all agents start with the same initial action $a$ or the same initial immunity $\tau$. Again, some numerical explorations revealed that for fixed $a$ and uniform $\tau$ or for fixed $\tau$ and uniform $a$, the contagion tends to spread throughout the population, yielding a rather interesting phenomenon. However, the current mathematical tools will fail to encompass such levels of randomness and proving the occurrence of the phenomena mentioned above rigorously may require completely different mathematical machinery. Finally, we would also like to explore the situation where $g_{ij}$ is allowed to change over time or have its own model of evolution. That would also bring significant changes to our computations and thus left for future persuasion.

\ignore{
Finding 1: The limiting distribution of the cardinality of the final set

Finding 2: Dynamic g has similar properties

Finding 3: Action set convergence 

Finding 4: Simulations 
 -a. When all a' are random it is spreading
 -b. For the regions we did not cover
 -c. What happens if taus are also random but maybe not the entire region of 0 to 1 with a fixed a?
 
}

\setcitestyle{numbers}
\newpage
\bibliography{diseasegraph}

\newpage
\begin{appendix}\label{sec:appendix}
\section{Proof of Lemma \ref{rem_0} and Lemma \ref{lem_1}} \label{sec:appendix_A}
\subsection{Proof of Lemma \ref{rem_0}}
\begin{proof} We provide the proof by distinguishing three cases as considered in the statement of the lemma. \\
	\textbf{Case 1.} $i \in I(S)$.\\
	By (\ref{ut}), $u_i\left(I(S),(a_i,a_{-i}(S))\right)=f(a_i)$. As $f$ is an increasing function, her best response is 
	$$b_i(S)= \argmax_{a_i\in [0,1]}f(a_i)=1.$$ 
	
	\noindent \textbf{Case 2.}  $i \notin S$ and $r_i(S)=0$. \\
	Since $r_i(S)=0$,  $a_ir_i(S) \leqslant \tau(i)$ for all $a_i\in [0,1]$, and hence by (\ref{ut}), $u_i\left(I(S),(a_i,a_{-i}(S))\right)=1+f(a_i)$ for all $a_i\in [0,1]$. This implies $b_i(S)=1$. \\
	\textbf{Case 3.}  $i \notin S$ and $r_i(S)>0$.\\
	Consider the quantity $\frac{\tau(i)}{r_i(S)}$. It follows from (\ref{ut}) that $u_i\left(I(S),(a_i,a_{-i}(S))\right)$ is increasing in $a_i$ in both the regions $\left[0,\min\left\{1,\frac{\tau(i)}{r_i(S)}\right\}\right]$ and $\left[\min\left\{1,\frac{\tau(i)}{r_i(S)}\right\},1\right]$. The maximum value of $u_i\left(I(S),(a_i,a_{-i}(S))\right)$ when $a_i$ lies in the region $\left[0,\min\left\{1,\frac{\tau(i)}{r_i(S)}\right\}\right]$  is $1+f\left(\min\left\{1,\frac{\tau(i)}{r_i(S)}\right\}\right)$ and that when $a_i$ lies  in the region $\left[\min\left\{1,\frac{\tau(i)}{r_i(S)}\right\},1\right]$ is $f(1)$. Because  $\tau(i) > 0$, we have $\frac{\tau(i)}{r_i(S)}> 0$. This, together with the fact that $f$ is strictly increasing, implies $f\left(\min\left\{1,\frac{\tau(i)}{r_i(S)}\right\}\right)> 0$. Hence, $1+f\left(\min\left\{1,\frac{\tau(i)}{r_i(S)}\right\}\right)>1$. Additionally, as $f(1)\leqslant 1$, we have $1+f\left(\min\left\{1,\frac{\tau(i)}{r_i(S)}\right\}\right)>f(1)$. Therefore, $u_i\left(I(S),(a_i,a_{-i}(S))\right)$ will be uniquely maximum at $\min\left\{1,\frac{\tau(i)}{r_i(S)}\right\}$ implying $b_i(S)=\min\left\{1,\frac{\tau(i)}{r_i(S)}\right\}$. Combining all these, we have the following form of $b_i(S)$.
	
	\[
	b_i(S)= 
	\begin{cases}
		1 & \text{if }  i \in I(S),\\
		1 & \text{if }  i \notin I(S) \text{ and } r_i(S)=0, \\
		\min \left\{1,\frac{\tau(i)}{r_i(S)}\right\} & \text{if }  i \notin I(S) \text{ and } r_i(S) \neq 0. \\
	\end{cases}
	\]
	
%

\end{proof}

\subsection{Proof of Lemma \ref{lem_1}}

\begin{proof}
 Let $\uti$ be a DVSP. It follows from the definition of $
 S(\uti)$ that $I(S_1(\uti))\subseteq I(S_2(\uti))\subseteq \cdots \subseteq I(S_t(\uti))$ for any $t\in \mathbb{N}_0$. As $I(S_t)\subseteq N$ for all $t$, this means $\lim_{t \to \infty} I(S_t(\uti))$  exists. Also, as $|N|$ is finite, there exists $t_0$ such that $I(S_{t_0}(\uti))=I(S_t(\uti))$ for all $t\geqslant t_0$. Consider $\bar{t}\geqslant t_0+1$. In the next claim, we show that for all $i\in N$, $a_i(S_{\bar{t}+1})\geqslant a_i(S_{\bar{t}})$.
 
 \begin{claim}\label{cl_3}
 	$a_i(S_{\bar{t}+1})\geqslant a_i(S_{\bar{t}})$ for all $i\in N$.
 \end{claim}
\noindent \textbf{Proof of the claim:} Let $\uti_{\bar{t}}=j$. By the definition of the process, for any other agent $i$ we have  $a_i(S_{\bar{t}+1})= a_i(S_{\bar{t}})$, and hence  the claim holds for them. We proceed  show that the claim holds for agent $j$. Recall that $a_j(\hat{S}_{{\bar{t}}})=a_j(S_{{\bar{t}+1}})$ (see Observation \ref{rem_0.25}). If $j\in I(S_{\bar{t}})$ then $a_j(\hat{S}_{\bar{t}})=1$ (see Observation \ref{obs_3}), and hence $a_j(\hat{S}_{\bar{t}})\geqslant a_j(S_{\bar{t}})$. 
 As $a_j(\hat{S}_{\bar{t}})=a_j(S_{\bar{t}+1})$, this means $a_j(S_{\bar{t}+1})\geqslant a_j(S_{\bar{t}})$. If $j\notin I(S_{\bar{t}})$, by the definition of the process, $j$ will choose her action as $a_j(\hat{S}_{\bar{t}})=b_j(S_{\bar{t}})$. If $b_j(S_{\bar{t}})=1$ then there is nothing to show. Assume $b_j(S_{\bar{t}})<1$. This implies  $a_j(\hat{S}_{\bar{t}})=\frac{\tau(j)}{r_j(S_{\bar{t}})}$. As $j\notin I(S_{\bar{t}})$, we have $a_j(\hat{S}_{{\bar{t}-1}})r_j(\hat{S}_{\bar{t}-1})\leqslant \tau(j)$. Also, as $\bar{t}\geqslant t_0+1$, it follows that $I(S_{\bar{t}-1})=I(S_{\bar{t}})$, which implies  $\hat{S}_{\bar{t}-1}=S_{\bar{t}}$ (see (iii) of Remark \ref{rem_0.25}) and hence $r_j(\hat{S}_{\bar{t}-1})=r_j(S_{\bar{t}})$. Combining this with the fact that $a_j(S_{{\bar{t}}})=a_j(\hat{S}_{{\bar{t}-1}})$, we obtain $a_j(S_{{\bar{t}}})\leqslant \frac{\tau(j)}{r_j(S_{\bar{t}})}$. Since $a_j(\hat{S}_{\bar{t}})=\frac{\tau(j)}{r_j(S_{\bar{t}})}$ and $a_j(S_{{\bar{t}}})\leqslant \frac{\tau(j)}{r_j(S_{\bar{t}})}$, we have  $a_j(\hat{S}_{{\bar{t}}})\geqslant a_j(S_{{\bar{t}}})$. This completes the proof of the claim. \hfill$\square$
 
 Since $a_i(S_t)\leqslant 1$ for all $i\in N$, by Claim \ref{cl_3}, we have the convergence of $a_{N}(S_t(\uti))$.
\end{proof}

\section{A few important lemmas}\label{sec:appendix_B}

\begin{lemma}\label{lem_2}
	Suppose $\uti\in N_\infty$ and let  $\bar{t}\in \mathbb{N}_0$ be such that either
	\[\uti_{\bt}\in I(S_{\bt}) \mbox{ and }a_{\uti_{\bt}}(S_{\bt})=1,\] or, 
		 \[\uti_{\bt}\notin I(S_{\bt}) \mbox{ and }I(S_{\bt-1})=I(S_{\bt}).\]
				 Then $I(S_{\bt})=I(S_{\bt+1})$.
\end{lemma}

\begin{proof}
	  First assume that $I(S_{\bt-1})=I(S_{\bt})$ and $\uti_{\bar{t}}=i$ with $i\notin I(S_{\bt})$. Since $i\notin I(S_{\bt})$, $i$ will choose her action as
	  \[
	   b_i(S_{\bt})= 
	   \begin{cases}
	    1 & \text{if }  r_i(S_{\bt})=0,\\
	    	\min \left\{1,\frac{\tau(i)}{r_i(S_{\bt})}\right\} & \text{if }  r_i(S_{\bt}) \neq 0.
	   \end{cases}
	  \]
	  
	  	
	 
	   If $r_i(S_{\bt})=0$ then $r_i(\hat{S}_{\bt})=r_i(S_{\bt})=0$, and agent $i$ will not get infected as $1\times r_i(\hat{S}_{\bt})=0\leqslant \tau(i)$. Suppose $r_i(S_{\bt})>0$. Since $a_i(\hat{S}_{{\bar{t}}})=b_i(S_{{\bar{t}}})$ and $r_i(S_{\bar{t}})=r_i(\hat{S}_{\bar{t}})$, this means agent $i$ will not get infected at $\bar{t}+1$. To show that  any other agent $j\notin I(S_{\bt})$ will not get infected at $\bar{t}+1$, we first claim that $a_i(\hat{S}_{{\bar{t}}})\geqslant a_i(S_{{\bar{t}}})$. If $a_i(\hat{S}_{{\bar{t}}})=1$ then there is nothing to show, so, assume $a_i(\hat{S}_{{\bar{t}}})=\frac{\tau(i)}{r_{i}(S_{\bar{t}})}$.  As $i\notin I(S_{\bar{t}})$, we have $a_i(\hat{S}_{{\bar{t}-1}})r_i(\hat{S}_{\bar{t}-1})\leqslant \tau(i)$.  Moreover, as $I(S_{\bar{t}-1})=I(S_{\bar{t}})$, it follows that $\hat{S}_{\bar{t}-1}=S_{\bar{t}}$ (see (iii) of Remark \ref{rem_0.25}) and hence, $r_i(\hat{S}_{\bar{t}-1})=r_i(S_{\bar{t}})$. Combining it with $a_i(S_{{\bar{t}}})=a_i(\hat{S}_{{\bar{t}-1}})$, we get $a_i(S_{{\bar{t}}})r_i(S_{\bar{t}})\leqslant \tau(i)$. So, $a_i(\hat{S}_{{\bar{t}}})\geqslant a_i(S_{{\bar{t}}})$.
	  
	  Take $j\in I(S_{\bt})$ with $j\neq i$. Since $j\notin I(S_{\bar{t}})$, it means $a_j(\hat{S}_{\bar{t}-1})r_j(\hat{S}_{\bar{t}-1})\leqslant \tau(j)$. Additionally, $j\neq i$ implies $a_j(\hat{S}_{\bar{t}-1})=a_j(S_{\bar{t}})=a_j(\hat{S}_{\bar{t}})$. Therefore, to show that $a_j(\hat{S}_{\bar{t}})r_j(\hat{S}_{\bar{t}})\leqslant \tau(j)$, it is enough to show $r_j(\hat{S}_{\bar{t}-1})\geqslant r_j(\hat{S}_{\bar{t}})$. Note that
	  \begin{align*}
	  	r_j(\hat{S}_{\bar{t}-1})&=\frac{\sum_{k\in I(\hat{S}_{\bar{t}-1})\setminus j}a_k(\hat{S}_{\bar{t}-1})g_{jk}}{\sum_{k\in N\setminus j}a_k(\hat{S}_{\bar{t}-1})g_{jk}} \\
	  	&=\frac{\sum_{k\in I(\hat{S}_{\bar{t}-1})\setminus j}a_k(\hat{S}_{\bar{t}})g_{jk}}{a_i(\hat{S}_{\bar{t}-1})g_{ji}+\sum_{k\in N\setminus \{i,j\}}a_k(\hat{S}_{\bar{t}})g_{jk}}  \text{ (as }i\notin I(\hat{S}_{\bar{t}-1}) \text{ and } a_k(\hat{S}_{\bar{t}-1})=a_k(\hat{S}_{\bar{t}}) \forall k\neq i)\\
	  	&=\frac{\sum_{k\in I(\hat{S}_{\bar{t}})\setminus j}a_k(\hat{S}_{\bar{t}})g_{jk}}{a_i(\hat{S}_{\bar{t}-1})g_{ji}+\sum_{k\in N\setminus \{i,j\}}a_k(\hat{S}_{\bar{t}})g_{jk}}  \text{ (as } I(\hat{S}_{\bar{t}-1})=I(S_{\bt-1})=I(S_{\bt})=I(\hat{S}_{\bar{t}}))\\
	  	&\geqslant \frac{\sum_{k\in I(\hat{S}_{\bar{t}})\setminus j}a_k(\hat{S}_{\bar{t}})g_{jk}}{a_i(\hat{S}_{\bar{t}})g_{ji}+\sum_{k\in N\setminus \{i,j\}}a_k(\hat{S}_{\bar{t}})g_{jk}}  \text{ (as } a_i(\hat{S}_{\bar{t}})\geqslant  a_i(S_{{\bar{t}}})=a_i(\hat{S}_{\bar{t}-1}))\\
	  	&=r_j(\hat{S}_{\bar{t}}).
	  \end{align*}
	  So, agent $j$ will not get infected at $\bar{t}+1$ and hence, $I(S_{\bar{t}+1})=I(S_{\bar{t}})$. 
	  
	  Now assume $i\in I(S_{\bt})$ with $a_{i}(S_{\bt})=1$. This means $a_i(\hat{S}_{\bt})=b_i(S_{\bt})=1$. As $b_i(S_{\bt})=a_{i}(S_{\bt})$ and $\uti_{\bt}=i$, we have $S_{\bt}=S_{\bt+1}$ (see Observation \ref{rem_0.25}). Hence, $I(S_{\bt})=I(S_{\bt+1})$. This completes the proof of the lemma.
	\end{proof}

\begin{lemma}\label{lem_3}
	Suppose that  $I(S_0)=\{1\}$ and $a_i(S_0)\leqslant \tau(i)$ for all $i\in N$. Let $\uti\in N_\infty$ and $\hat{t}\in \mathbb{N}_0$ be such that $\uti_t\neq 1$ for all $t<\hat{t}$. Then, $I(S_t)=\{1\}$ for all $t\leqslant \hat{t}$.
\end{lemma}

\begin{proof}
	Note that if $\hat{t}=0$ then there is nothing to show. So, assume $\hat{t}\geqslant 1$. We use induction to prove this. As the base case, we show that $I(S_1)=\{1\}$. Let $\uti_0=i$. Since $\hat{t}\geqslant 1$, $i\neq 1$. Agent $i$ will choose her action as
	  \[
	   b_i(S_0)= 
	   \begin{cases}
	    1 & \text{if }  r_i(S_0)=0,\\
	    	\min \left\{1,\frac{\tau(i)}{r_i(S_{0})}\right\} & \text{if }  r_i(S_0) \neq 0.
	   \end{cases}
	  \]	
		
	 If $r_i(S_0)=0$ then $r_i(\hat{S}_0)=r_i(S_0)=0$, and agent $i$ will not get infected as $1\times r_i(\hat{S}_0)=0\leqslant \tau(i)$. Suppose $r_i(S_0)>0$. Since $a_i(\hat{S}_0)=b_i(S_0)$, $r_i(\hat{S}_0)=r_i(S_0)$, and $b_i(S_0)\leqslant \frac{\tau(i)}{r_i(S_0)}$, this means agent $i$ will not get infected at $t=1$. For any $j\notin \{1,i\}$, $a_j(\hat{S}_0)=a_j(S_0)\leqslant \tau(j)$, so, agent $j$ will also not get infected at $t=1$. Thus, $I(S_1)=\{1\}$. Next we introduce an induction hypothesis.
	
	\noindent\textit{Induction Hypothesis:} Given $\bar{t}\in \mathbb{N}_0$ with $\hat{t}\geqslant \bar{t}>1$, we have $I(S_1)=\cdots=I(S_{\bar{t}-1})=\{1\}$.
	
	  We show that $I(S_{\bar{t}})=\{1\}$. Let $\uti_{\bar{t}-1}=i$. Since $\hat{t}\geqslant \bar{t}$, this means $i\neq 1$. Hence, $i\notin I(S_{\bt-1})$. As $\bt>1$, we have $I(S_{\bar{t}-2})=I(S_{\bt-1})$. This together with Lemma \ref{lem_2}, implies $I(S_{\bt})=I(S_{\bt-1})=\{1\}$. Thus, by induction, we have $I(S_{\hat{t}})=\{1\}$. This completes the proof of the lemma.  \end{proof}

     \begin{remark}\label{rem_1}
  	It follows from Lemma \ref{lem_3} that  $I(S_{t_1(\uti)})=\{1\}$ for all $\uti \in N_\infty$. 
  \end{remark}

\begin{lemma}\label{lem_4}
	 Consider $\uti\in N_\infty$ and let $\h \in \mathbb{N}_0$ be such that 
	$a_i(S_{\h})=1 \text{ for all } i\in I(S_{\h}) \text{ and } a_i(S_{\h})\leqslant \tau(i) \text{ for all } i\notin I(S_{\h})$. Then, $I(S_{\h})=I(S_\infty).$
\end{lemma}

\begin{proof}
	We first show that $I(S_{\h+1})=I(S_{\h})$. Let $\uti_{\h}=i$. Suppose $i\in I(S_{\h})$.  Thus by Lemma  \ref{rem_0}. $a_i(\hat{S}_{\h})=b_i(S_{\h})=1$. This implies $S_{\h}=S_{\h+1}$  (see Observation \ref{rem_0.25}) and hence, $I(S_{\h+1})=I(S_{\h})$. Now suppose  $i\notin I(S_{\h})$. Agent $i$ will choose her action as

	  \[
	   b_i(S_{\h})= 
	   \begin{cases}
	    1 & \text{if }  r_i(S_{\h})=0,\\
	    	\min \left\{1,\frac{\tau(i)}{r_i(S_{\h})}\right\} & \text{if }  r_i(S{\h}) \neq 0.
	   \end{cases}
	  \]

		

	If $r_i(S_{\h})=0$ then $r_i(\hat{S}_{\h})=r_i(S_{\h})=0$, and agent $i$ will not get infected as $a_i(\hat{S}_{\h})\times r_i(\hat{S}_{\h})=0\leqslant \tau(i)$. Suppose $r_i(S_{\h})>0$. Since $a_i(\hat{S}_{\h})=b_i(S_{\h})$ and $r_i(S_{\h})=r_i(\hat{S}_{\h})$, this means agent $i$ will not get infected at $\h+1$. Take $j\notin I(S_{\h})$ and $j\neq i$. Note that by the assumption of the lemma, $a_j(S_{\h})\leqslant \tau(j)$. Since $j\neq i$, we have $a_j(S_{\h})=a_j(\hat{S}_{\h})$. Combining these two, we have $a_j(\hat{S}_{\h})\leqslant \tau(j)$. As $r_j(\hat{S}_{\h})\leqslant 1$ this implies $a_j(\hat{S}_{\h})r_j(\hat{S}_{\h})\leqslant \tau(j)$. Thus, agent $j$ will not get infected at $\h+1$. Hence, $I(S_{\h+1})=I(S_{\h})$

	We now show that for any $t\in \mathbb{N}_0$ with $t>\h+1$, $I(S_{\h})=I(S_t)$ holds. Assume for contradiction there exists $\bt\in \mathbb{N}_0$ with $\bt>\h+1$ such that $I(S_{\h})\subsetneq I(S_{\bt})$. Without loss of generality we can assume that $I(S_{\h})=I(S_{\h+1})=\cdots=I(S_{\bt-1})$. Let $\uti_{\bt-1}=i$. Suppose $i\in I(S_{\bt-1})$.   We first show 
	\begin{equation}\label{eq_1}
		a_i(S_{\bt-1})=1.
	\end{equation}
	As $i\in I(S_{\bt-1})$ and $I(S_{\bt-1})=I(S_{\h})$, we have $i\in I(S_{\h})$. Thus, by the assumption of the lemma, $a_i(S_{\h})=1$. Since $\h\leqslant \bt-1$, this implies $a_i(S_{\bt-1})=1$; see Observation \ref{obs_3}.

	
	By (\ref{eq_1}), we have $a_i(S_{\bt-1})=1$. Since $\uti_{\bt-1}=i$ and $i\in I(S_{\bt-1})$, this implies $a_i(\hat{S}_{\bt-1})=1$. Thus, $S_{\bt-1}=\hat{S}_{\bt-1}$, and hence, $I(S_{\bt-1})=I(S_{\bt})$ (see Observation \ref{rem_0.25}), a contradiction to $I(S_{\h})\subsetneq I(S_{\bt})$. Hence, $I(S_{\bt})=I(S_{\h})$.
	
	Now suppose $i\notin I(S_{\bt-1})$. As $\bt>\h+1$, we have $I(S_{\bt-1})=I(S_{\bt-2})$. This together with Lemma \ref{lem_2} implies $I(S_{\bt-1})=I(S_{\bt})$, a contradiction to $I(S_{\h})\subsetneq I(S_{\bt})$. Hence, $I(S_{\bt})=I(S_{\h})$. This completes the proof of the lemma.
\end{proof}

\section{Proof of Theorem \ref{theo_1}, Theorem \ref{theo_1.5}, Theorem \ref{theo_2}, and Theorem \ref{theo_3}}  \label{sec:appendix_C}

\subsection{Proof of Theorem \ref{theo_1}}\label{subsec:proof_theo_1}
\begin{proof} 
	We complete the proof in two steps. In Step 1, we explore how the infection spreads when agents update their actions according to a fixed agent sequence, and in Step 2 we use this to explore how infection spreads when agents update their actions randomly.
	
	\noindent \textbf{Step 1.}    Fix an agent sequence  $\uti\in N_\infty$ and let $S$ be the DVSP induced by $\undertilde{v}$. To shorten notation, for all $i \in N$, let us denote $t_i(\uti)$ by $k_i$.  The following  claim demonstrates how an agent $i$ with $k_i<k_1$ will update her action. 
	
	\vspace{2mm}
	\noindent\textbf{Claim 1}: Suppose $k_i<k_1$ for some $i\in N$. Then, $a_i(S_t)=1$ for all $t=k_i+1,\ldots, k_1$.
	
	\noindent \textbf{Proof of the claim.}	By Lemma \ref{lem_3}, $I(S_t)=\{1\}$ for all $t\leqslant k_1$. Since $k_1 < \infty$ and $k_i<k_1$, we have  $k_i<\infty$. Consider any time point $l$ such that $k_i\leqslant l< k_1$. By the definition  of the process, we need to show that the Claim holds for $l$ such that $\uti_{l}=i$ (see Observation \ref{rem_0.5}). Since $l< k_1$, we have $a_1(S_{l})=a_1(S_0)=0$. This together with $I(S_{l})=\{1\}$ implies   $r_i(S_{l})=0$. Hence, by Remark \ref{rem_0.25}, agent $i$ will update her action to $1$ at $\hat{S}_{l}$. Since $a_i(\hat{S}_l)=a_i(S_{l+1})$, this means, $a_i(S_{l+1})=1$. This completes the proof of the Claim. \hfill$\square$

	\noindent\textbf{Case 1:} $|N_1(\uti)|\geqslant \alpha$. \\
	As $|N_1(\uti)|\leqslant n-1$, the assumption of the case implies $\alpha=\left \lceil\frac{1}{\tau}\right \rceil$. Hence, $\alpha\tau\geqslant 1$. By Claim 1,  $a_i(S_{k_1})=1$ for all  $i\in N_1(\uti)$. Also, by the definition of the process, $a_i(S_{k_1})=0$ for all  $i\notin N_1(\uti)\cup \{1\}$ as they have not updated their actions till the time point $k_1$. Recall that $\hat{S}_{k_1}$ denotes the intermediate state where the only change from $S_{k_1}$ is that agent $\uti_{k_1}$ has updated her action to $b_{\uti_{k_1}}(S_{k_1})$. Since $\uti_{k_1}=1$, we have  $a_i(S_{k_1})=a_i(\hat{S}_{k_1})$ for all $i\neq 1$. Thus, $a_i(\hat{S}_{k_1})=1$ for all $i\in N_1(\uti)$  and $a_i(\hat{S}_{k_1})=0$ for all $i\notin N_1(\uti)\cup \{1\}$.
	
	By Remark \ref{rem_0} and the definition of the process,  $a_1(\hat{S}_{k_1})=1$. Consider the time point $k_1+1$. By the definition of the process, an agent $i\neq 1$ will be in $I(S_{k_1+1})$ if $a_i(\hat{S}_{k_1})r_i(\hat{S}_{k_1})>\tau$. Since $I(S_{k_1})=\{1\}$, $a_i(\hat{S}_{k_1})=1$ for all $i\in N_1(\uti) \cup \{1\}$, and $a_i(\hat{S}_{k_1})=0$ for all $i\notin N_1(\uti)\cup \{1\}$, it follows that $r_i(\hat{S}_{k_1})\leqslant \frac{1}{\alpha}$ for all $i\in N_1(\uti)$. Because $\alpha\tau\geqslant 1$, this implies that no agent in $N_1(\uti)$  gets infected at the time point $k_1+1$. Moreover, since  $a_i(\hat{S}_{k_1})=0$ for each agent $i\notin N_1(\uti)\cup \{1\}$, we have  $a_i(\hat{S}_{k_1})r_i(\hat{S}_{k_1})=0\leqslant \tau$.  Thus, no new agent gets infected at the time point $k_1+1$, and hence, $I(S_{k_1+1})=\{1\}$.
	
	We show that no new agent would get infected after this. We first show that $I(S_{k_1+2})=\{1\}$. Let $\uti_{k_1+1}=i$. If $i\notin I(S_{k_1+1})$ then as $I(S_{k_1})=I(S_{k_1+1})$ by Lemma \ref{lem_2}, we have $I(S_{k_1+1})=I(S_{k_1+2})$. If $i\in I(S_{k_1+1})$ then $i=1$. Moreover, $a_1(S_{k_1+1})=a_1(\hat{S}_{k_1})=1$. Hence, by Lemma \ref{lem_2}, $I(S_{k_1+1})=I(S_{k_1+2})$. Therefore, $I(S_{k_1+2})=\{1\}$. Using the same arguments repeatedly, it follows that $I(S_t)=\{1\}$ for all $t\geqslant k_1+2$. Thus,  $I(S_\infty)=\{1\}$.

	\noindent\textbf{Case 2:} $|N_1(\uti)|\leqslant \alpha-1$. \\
	Using similar arguments as in Case 1, we have $a_i(\hat{S}_{k_1})=1$ for all $i\in N_1(\uti)$ and $a_i(\hat{S}_{k_1})=0$ for all $i\notin N_1(\uti)\cup \{1\}$. This implies  $r_i(\hat{S}_{k_1})\geqslant\frac{1}{\alpha-1}$ for all $i\in N_1(\uti)$. As $\alpha=\min\left\{\left \lceil\frac{1}{\tau}\right \rceil,n\right\}$, we have $(\alpha-1)\tau<1$. Hence, all agents in $N_1(\uti)$ will get infected at time point $k_1+1$. Moreover, as $a_i(\hat{S}_{k_1})=0$ for all $i\notin N_1(\uti)\cup \{1\}$, the agents outside $N_1(\uti) \cup \{1\}$ will not get infected at time point $k_1+1$. Thus, we have  $I(S_{k_1+1})=N_1(\uti)\cup \{1\}$. Because,  $a_i(S_{k_1+1})=a_i(\hat{S}_{k_1})=1$  for all $i\in I(S_{k_1+1})$ and $a_i(S_{k_1+1})=0\leqslant \tau$ for all $i\notin I(S_{k_1+1})$,  by Lemma \ref{lem_4} it follows that  $I(S_{k_1+1})=I(S_\infty)$. Hence, $I(S_\infty)=N_1(\uti)\cup \{1\}$.

	\noindent \textbf{Step 2.} Consider the probability space $(N_\infty, \mathcal{F}, \mathbb{P})$ and random variables $S$ and $t_1,\ldots, t_n$. Let $m \in \{2,\ldots, n\}$ be such that $m \leqslant \alpha$.  In view Case 1 and Case 2 of the current proof, we have  (i) $|I(S_\infty)|\leqslant \alpha$, and (ii) $|I(S_\infty)| = m$ with $1\in I(S_\infty)$ if and only if  $| \{i \in N \mid t_i < t_1\}| = m-1$. Also, $I(S_\infty) = \{1\}$ if and only if $ \{i \in N \mid t_i < t_1\}=\emptyset$. Moreover, as $\mathbb{P}$ is uniform, any two subsets of $N$ with same cardinality have the same probability.
	These observations together yield  
	
\[
	\mathbb{P}(I(S_\infty)=J)= 
	\begin{cases} 
		1-\frac{\alpha-1}{n} & \text{if } J=\{1\} , \\\\
		\frac{1}{n \times \Mycomb[n-1]{m-1}} & \text{if } 1\in J \text{ and } |J|=m \text{ where } m\in[2,\alpha], 
		\\\\ 
		0 & \text{ otherwise }.
	\end{cases}
	\]
	
		This completes the proof of the theorem. 
\end{proof}

\subsection{Proof of Theorem \ref{theo_1.5}}

\begin{proof} 
	We follow the same structure that we used in the proof of Theorem \ref{theo_1}. 
	
	\noindent \textbf{Step 1.}    Fix an agent sequence  $\uti\in N_\infty$ and let $S$ be the DVSP induced by $\undertilde{v}$. To shorten notation, for all $i \in N$, let us denote $t_i(\uti)$ by $k_i$.  The following  claim demonstrates how an agent $i$ with $k_i<k_1$ will update her action. Recall the set $N_1(\uti)$. We distinguish two cases based on the value of $|N_1(\uti)|$. \\

	

	\noindent\textbf{Case 1:} $|N_1(\uti)|=0$. \\
First assume $\tau\geqslant \frac{1}{n-1}$. We show that no agent will get infected under this assumption, i.e., $I(S_\infty)=\{1\}$. Note that by the assumption of the case, $\uti_0=1$. Also, as $a=1$,  $a_i(S_{0})=1$ for all  $i\in N$.  Recall that $\hat{S}_{0}$ denotes the intermediate state where the only change from $S_{0}$ is that agent $\uti_{0}$ has updated her action to $b_{\uti_{0}}(S_{0})$. Since $\uti_{0}=1$, we have  $a_i(S_{0})=a_i(\hat{S}_{0})$ for all $i\neq 1$. Thus, $a_i(\hat{S}_{0})=1$ for all $i\in N\setminus \{1\}$. Moreover, by Remark \ref{rem_0} and the definition of the process,  $a_1(\hat{S}_{0})=1$. Consider the time point $1$. By the definition of the process, an agent $i\neq 1$ will be in $I(S_{1})$ if $a_i(\hat{S}_{0})r_i(\hat{S}_{0})>\tau$. Since $I(S_{0})=\{1\}$, $a_i(\hat{S}_{0})=1$ for all $i\in N$, it follows that $r_i(\hat{S}_{0})= \frac{1}{n-1}$ for all $i\in N\setminus \{1\}$. Because $\tau\geqslant \frac{1}{n-1}$, this implies that no agent in $N\setminus \{1\}$  gets infected at the time point $1$. Hence, $I(S_{1})=\{1\}$.
	
	We now show that no new agent would get infected after this. We first show that $I(S_{2})=\{1\}$. Let $\uti_{1}=i$. If $i\notin I(S_{1})$ then as $I(S_{0})=I(S_{1})$ by Lemma \ref{lem_2}, we have $I(S_{1})=I(S_{2})$. If $i\in I(S_{1})$ then $i=1$. Moreover, $a_1(S_{1})=a_1(\hat{S}_{0})=1$. Hence, by Lemma \ref{lem_2}, $I(S_{1})=I(S_{2})$. Therefore, $I(S_{2})=\{1\}$. Using the same arguments repeatedly, it follows that $I(S_t)=\{1\}$ for all $t\geqslant 2$. Thus,  $I(S_\infty)=\{1\}$.

 Now assume $\tau< \frac{1}{n-1}$. We show that all the agent gets infected under this assumption. Using similar arguments as before, we get $r_i(\hat{S}_{0})= \frac{1}{n-1}$ for all $i\in N\setminus \{1\}$. As $\tau< \frac{1}{n-1}$, this means each $i\in N\setminus \{1\}$ will get infected  at time point $1$. Therefore,  $I(S_\infty)=N$. \\

	\noindent\textbf{Case 2:} $|N_1(\uti)|\geqslant  1$.\\ 
This means $\uti_0\neq 1$. Let $\uti_0=i\notin \{1\}$. Hence, by the definition of the process, agent $i$ will choose her action as $b_i(S_0)$ at the intermediate state $\hat{S}_0$. As $a_j(S_0)=1$ for all $j\in N$ and $I(S_0)=\{1\}$, it follows that $r_i(S_0)\neq 0$. Therefore,
 \begin{equation}\label{e_1}
 b_i(S_0)=\min \left\{1,\frac{\tau}{r_i(S_0)}\right \}=\min \left \{1,(n-1)\tau\right \}.    
 \end{equation}

\noindent  Since by our assumption $\uti_{0}=i$ and $i\notin I(S_{0})$, by Observation \ref{obs_4}, $i\notin I(S_{1})$. For any other uninfected agent $j$, $$r_j(\hat{S}_0)=\frac{1}{(n-2)+b_i(S_0)}.$$ This together with the fact that $a_j(\hat{S}_0)=1$ implies 
 \begin{enumerate}
     \item if $\tau\geqslant \frac{1}{n-1}$ then $b_i(S_0)=1$ and hence, $a_j(\hat{S}_0)r_j(\hat{S}_0)=\frac{1}{n-1}\leqslant \tau$, and
     \item if $\tau< \frac{1}{n-1}$ then $b_i(S_0)<1$ and hence, $a_j(\hat{S}_0)r_j(\hat{S}_0)>\frac{1}{n-1}>\tau$.
 \end{enumerate}
 Combining the above observations, we may write if $\tau\geqslant \frac{1}{n-1}$ then agent $j$ will not get infected at time point $1$ and if $\tau< \frac{1}{n-1}$ then agent $j$ will get infected at time point $1$.  Hence, we have
 $$\tau\geqslant \frac{1}{n-1} \implies I(S_1)=\{1\} \mbox{ and } \tau< \frac{1}{n-1} \implies I(S_1)=N\setminus \{i\}.$$

 If $\tau\geqslant \frac{1}{n-1}$ then using similar arguments as in Case 1, we can show that $I(S_\infty)=\{1\}$. If $\tau< \frac{1}{n-1}$ then to decide the final infected set we distinguish two subcases. \\

 \noindent \textbf{Case 2.1.}  $\uti_1=i$.\\
 We show that the final infected set will be $N\setminus i$. Since by our assumption $\uti_{1}=i$ and $i\notin I(S_{1})$, by Observation \ref{obs_4}, $i\notin I(S_{2})$. Hence, $I(S_{2})=N\setminus \{i\}$. We now show that $i$ will not get infected after this. At time point $2$,
	$$r_i(\hat{S}_{2})=\frac{(n-1)}{(n-1)}=1.$$ Therefore, $a_i(\hat{S}_{2})=\tau$ (see Observation \ref{obs_4}). At time point $3$, if $\uti_{3}=i$, then agent $i$ would not get infected at time point $4$ (Observation \ref{obs_4}). On the other hand, if $\uti_{3}\neq i$ then as $a_i(\hat{S}_{3})=a_i(\hat{S}_{2})=\tau$, it follows that $a_i(\hat{S}_{3})r_i(\hat{S}_3)\leqslant \tau$. Hence, agent $i$ would remain uninfected at time point $4$. Continuing in this manner, we may show that $i$ will  not get infected after this. Thus, $I(S_\infty)=N\setminus \{i\}$.

	\vspace{2mm}
	\noindent\textbf{Case 2.2.:} $\uti_{1}\neq i$ \\
	We show that the final infected set will be $N$. Since  $I(S_{1})=N\setminus \{i\}$, $r_i(\hat{S}_{1})=1$. Moreover, as $a_i(S_{1})=a_i(\hat{S}_{0})=b_i(S_0)=(n-1)\tau>\tau$ (see \ref{e_1}) and $\uti_{1}\neq i$, it follows that $a_i(\hat{S}_{1})>\tau$. Combining this two we have $a_i(\hat{S}_{1})r_i(\hat{S}_{1})>\tau$. Thus, agent $i$ will get infected at time point $2$. Hence, $I(S_{2})=N$ and  $I(S_\infty)=N$. \\
	
	\noindent \textbf{Step 2.} First assume $\tau\geqslant \frac{1}{n-1}$. Therefore, in view Case 1 and Case 2 of the current proof, we have $I(S_\infty)=\{1\}$.

 Now assume $\tau<\frac{1}{n-1}$. By Case 1 and Case 2 above, we have 
	\begin{enumerate}[(i)]
		\item $|I(S_\infty)|=n-1$ with $1\in I(S_\infty)$  if $|N_1(\uti)|\geqslant 1$ and there is $i\in N\setminus \{1\}$ such that $k_i=0$ and $\uti_{1}=i$, and
		\item $I(S_\infty)=N$  if either $|N_1(\uti)|=0$ or $|N_1(\uti)|\geqslant 1$ and there is no $i\in N\setminus \{1\}$ such that $k_i=0$ and $\uti_{1}=i$. 
	\end{enumerate}

 	We calculate the probability of $|I(S_\infty)|=n-1$. By (i) we have
	\begin{align*}
		&P(\uti\mid |N_1(\uti)|\geqslant 1 \text{ and } \exists i\neq 1 \text{ such that } k_i=0 \text{ and } \uti_{1}=i)\\
          =&P(\uti\mid  \exists i\neq 1 \text{ such that } k_i=0 \text{ and } \uti_{1}=i)\\
		=&\Mycomb[n-1]{1}\times \frac{1}{n^2}\\
           =&\frac{n-1}{n^2}.
	\end{align*}
Note that by (i) and (ii), 
$$P(|I(S_\infty)|=n-1)+P(I(S_\infty)=N)=1.$$ Therefore,
	\begin{align*}
		P(I(S_\infty)=N)&=1-P(|I(S_\infty)|=n-1)\\
		&=1-\frac{n-1}{n^2}. 
	\end{align*}
Since any two subsets of $N$ with the cardinality $n-1$ have the same probability,	combining all the above observations, we have the following distribution of the infected set. 

\[
 \mathbb{P}(I(S_\infty)=J)= 
  \begin{cases} 
   \frac{1}{n^2} & \text{if } 1\in J \text{ and } |J|=n-1,
   \\ \\
    1-\frac{n-1}{n^2} & \text{if } |J|=n, \text{ i.e., } J=N, \\ \\
   0 & \text{ otherwise. } 
  \end{cases}
\]    
This completes the proof of the theorem. 
\end{proof}

\subsection{ Proof of Theorem \ref{theo_2}}
\begin{proof}
	Note that 
	\begin{align}\label{eq_10}
		\left[\halpha\geqslant \left\lceil \frac{\frac{1}{\tau}-(n-1)a}{1-a}\right\rceil\right]
		\iff \left[\tau\geqslant \frac{1}{\halpha+(n-1-\halpha)a} \right], 
	\end{align}
	and 
	\begin{align}\label{eq_11}
		\left[\halpha= \left\lceil \frac{\frac{1}{\tau}-(n-1)a}{1-a}\right\rceil\right]
		\iff \left[\frac{1}{\halpha-1+(n-\halpha)a}>\tau\geqslant \frac{1}{\halpha+(n-1-\halpha)a} \right]. 
	\end{align}
	Also, as $\tau\geqslant \frac{1}{n-1}$, we have $\halpha\leqslant n-1$. We follow the same structure that we used in the proof of Theorem \ref{theo_1}. 
 
 \noindent \textbf{Step 1} Fix an agent sequence  $\uti\in N_\infty$ and let $S$ be the virus spread process induced by $\undertilde{v}$. To shorten notation, for all $i \in N$, let us denote $t_i(\uti)$ by $k_i$. We first prove a claim similar to Claim 1 as in Step 1 of the proof of Theorem \ref{theo_1}.
	
	\vspace{2mm}
	\noindent\textbf{Claim 1}: For all $0\leqslant t<k_1$, $a_i(S_{t+1})=1$ where $\uti_t=i$.
	
	\noindent \textbf{Proof of the claim.}  Let $\uti_0=i$. As $k_1>0$, $i\neq 1$. Since $a_j(S_0)=a>0$ for all $j\in N$ and $I(S_0)=\{1\}$, we have $r_i(S_0)=\frac{1}{(n-1)}$. This means $$b_i(S_0)=\min\left\{1,\frac{\tau}{\frac{1}{(n-1)}}\right\}=\min\{1,(n-1)\tau\}=1$$ as by the assumption of the lemma $\tau\geqslant \frac{1}{(n-1)}$. Thus, $a_i(S_1)=a_i(\hat{S}_0)=1$. Next we introduce an induction hypothesis.
	
	\noindent\textit{Induction Hypothesis:} Given $\bar{t}\in \mathbb{N}_0$ with $\bt<k_1$, we have for all $t<\bt$, $a_j(S_{t+1})=1$ where $\uti_t=j$.

	Let $\uti_{\bt}=i'$ and we show that $a_{i'}(S_{\bt+1})=1$. Note that by Lemma \ref{lem_3}, $I(S_{\bt})=\{1\}$. Moreover, by the induction hypothesis, $a_j(S_{\bt})\geqslant a$ for all $j\in N\setminus \{1\}$. Also, as $\bt<k_1$, we have $a_1(S_{\bt})=a$. Combining all these observations we have, 
	\begin{equation}\label{eq_0}
		\frac{1}{(n-1)}\geqslant r_{i'}(S_{\bt})\geqslant \frac{a}{(n-1)}.
	\end{equation}
	Since $r_{i'}(S_{\bt})>0$, $b_{i'}(S_{\bt})=\min\left\{1,\frac{\tau}{r_{i'}(S_{\bt})}\right\}$; see Remark \ref{rem_0}. Therefore, using (\ref{eq_0}) and the fact $\tau\geqslant \frac{1}{(n-1)}$, we have  $b_{i'}(S_{\bt})=1$. Thus, $a_{i'}(S_{\bt+1})=a_{i'}(\hat{S}_{\bt})=1$. This completes the proof of the claim. \hfill$\square$
	
	We now determine the final infection set. Note that by Claim 1,  $a_i(S_{k_1})=1$ for all  $i\in N_1(\uti)$. Also, by the definition of the process, $a_i(S_{k_1})=a$ for all  $i\notin N_1(\uti)\cup \{1\}$ as they have not updated their actions till the time point $k_1$. Recall that $\hat{S}_{k_1}$ denotes the intermediate state where the only change from $S_{k_1}$ is that agent $\uti_{k_1}$ has updated her action to $b_{\uti_{k_1}}(S_{k_1})$. Since $\uti_{k_1}=1$, we have  $a_i(S_{k_1})=a_i(\hat{S}_{k_1})$ for all $i\neq 1$. Thus, $a_i(\hat{S}_{k_1})=1$ for all $i\in N_1(\uti)$  and $a_i(\hat{S}_{k_1})=a$ for all $i\notin N_1(\uti)\cup \{1\}$.
	
	Moreover, by Remark \ref{rem_0} and the definition of the process,  $a_1(\hat{S}_{k_1})=1$. Consider the time point $k_1+1$. By the definition of the process, an agent $i\neq 1$ will be in $I(S_{k_1+1})$ if $a_i(\hat{S}_{k_1})r_i(\hat{S}_{k_1})>\tau$. For any $i\notin N_1(\uti)\cup 1$, $a_i(\hat{S}_{k_1})=a\leqslant \tau$. Thus, $a_i(\hat{S}_{k_1})r_i(\hat{S}_{k_1})\leqslant \tau$ and any agent in 
	$ N_1(\uti)\cup \{1\}$ will not get infected at $k_1+1$. For agents in $N_1(\uti)$, we distinguish two cases.
	
	\noindent\textbf{Case 1:} $|N_1(\uti)|\geqslant \halpha$. \\
	Since $I(S_{k_1})=\{1\}$, $a_i(\hat{S}_{k_1})=1$ for all $i\in N_1(\uti) \cup \{1\}$, and $a_i(\hat{S}_{k_1})=a$ for all $i\notin N_1(\uti)\cup \{1\}$, it follows that 
	\begin{align*}
		r_i(\hat{S}_{k_1})&=\frac{1}{|N_1(\uti)|+(n-1-|N_1(\uti)|)a} \\
		&=\frac{1}{|N_1(\uti)|(1-a)+(n-1)a} \\
		&\leqslant \frac{1}{\halpha(1-a)+(n-1)a} \hspace{20mm}\text{ (since } |N_1(\uti)|\geqslant \halpha)\\
		&\leqslant \tau \hspace{53mm}\text{ (by (\ref{eq_10}))}
	\end{align*}
	for all $i\in N_1(\uti)$. This implies that no agent in $N_1(\uti)$  gets infected at the time point $k_1+1$. 

	We show that no new agent would get infected after this. We first show that $I(S_{k_1+2})=\{1\}$. Let $\uti_{k_1+1}=i$. If $i\notin I(S_{k_1+1})$ then as $I(S_{k_1})=I(S_{k_1+1})$ by Lemma \ref{lem_2}, we have $I(S_{k_1+1})=I(S_{k_1+2})$. If $i\in I(S_{k_1+1})$ then $i=1$. Moreover, $a_1(S_{k_1+1})=a_1(\hat{S}_{k_1})=1$. Hence, by Lemma \ref{lem_2}, $I(S_{k_1+1})=I(S_{k_1+2})$. Therefore, $I(S_{k_1+2})=\{1\}$. Using the same arguments repeatedly, it follows that $I(S_t)=\{1\}$ for all $t\geqslant k_1+2$. Thus,  $I(S_\infty)=\{1\}$.

	\noindent\textbf{Case 2:} $|N_1(\uti)|\leqslant \halpha-1$. \\
By the assumption of the case, $\halpha\geqslant 1$. First assume $\halpha=1$. This, together with $|N_1(\uti)|\leqslant \halpha-1$, implies $|N_1(\uti)|=0$. Therefore, $k_1=1$. We show that $I(S_\infty)=\{1\}$. Note that by the definition of the process,  $a_i(\hat{S}_{0})=a$ for all $i\neq 1$. As $a\leqslant \tau$, this means no agent in the set $\{2,\ldots,n\}$ will get infected at the time point $1$. Hence, $I(S_1)=\{1\}$. Moreover, as $I(S_1)=\{1\}$ with $a_1(S_1)=1$ and $a_i(S_1)=a\leqslant \tau$ for all $i\neq 1$,  by Lemma \ref{lem_4} it follows that  $I(S_{1})=I(S_\infty)$. Hence, $I(S_\infty)=\{1\}$.

	Now assume $\halpha\geqslant 2$. Thus by the definition of $\halpha$, we have  $\halpha= \left\lceil \frac{\frac{1}{\tau}-(n-1)a}{1-a}\right\rceil$. As $I(S_{k_1})=1$, $a_i(\hat{S}_{k_1})=1$ for all $i\in N_1(\uti)\cup \{1\}$, and $a_i(\hat{S}_{k_1})=a$ for all $i\notin N_1(\uti)\cup \{1\}$, we have for all $i\in N_1(\uti)$
	\begin{align*}
		r_i(\hat{S}_{k_1})&=\frac{1}{|N_1(\uti)|+(n-1-|N_1(\uti)|)a} \\
		&=\frac{1}{|N_1(\uti)|(1-a)+(n-1)a} \\
		&\geqslant \frac{1}{(\halpha-1)(1-a)+(n-1)a} \hspace{20mm}\text{ (since } |N_1(\uti)|\leqslant (\halpha-1))\\
		&=  \frac{1}{(\halpha-1)+(n-\halpha)a}\\
		&>\tau. \hspace{63mm}\text{ (by (\ref{eq_11}))}
	\end{align*}
	This implies all agents in $N_1(\uti)$ will get infected at time point $k_1+1$. Thus, we have  $I(S_{k_1+1})=N_1(\uti)\cup \{1\}$. Further, as,  $a_i(S_{k_1+1})=a_i(\hat{S}_{k_1})=1$  for all $i\in I(S_{k_1+1})$ and $a_i(S_{k_1+1})=a\leqslant \tau$ for all $i\notin I(S_{k_1+1})$,  by Lemma \ref{lem_4} it follows that  $I(S_{k_1+1})=I(S_\infty)$. Hence, $I(S_\infty)=N_1(\uti)\cup \{1\}$.

	\noindent \textbf{Step 2.} Consider the probability space $(N_\infty, \mathcal{F}, \mathbb{P})$ and random variables $S$ and $t_1,\ldots, t_n$. Let $m \in \{2,\ldots, n\}$ be such that $m \leqslant \halpha$.  In view of Case 1 and Case 2 of the current proof, we have  (i) $|I(S_\infty)|\leqslant \halpha$, and (ii) $|I(S_\infty)| = m$  with $1\in I(S_\infty)$ if and only if  $| \{i \in N \mid t_i < t_1\}| = m-1$. Also, $|I(S_\infty)| = 1$ if and only if either $\{i \in N \mid t_i < t_1\}=\emptyset$ or $| \{i \in N \mid t_i < t_1\}| \geqslant \halpha$. Moreover, as $\mathbb{P}$ is uniform, any two subsets of $N$ with same cardinality have the same probability.
	These observations together yield  

	\[
	\mathbb{P}(I(S_\infty)=J)= 
	\begin{cases} 
		1-\frac{\hat{\alpha}-1}{n} & \text{if } J=\{1\}  , \\ \\
		\frac{1}{n \times \Mycomb[n-1]{m-1}} & \text{if }  1\in J \text{ and } |J|=m \text{ where } m\in[2, \halpha],
		\\ \\
		0 & \text{ otherwise.}
	\end{cases}
	\]
This completes the proof of the theorem. \end{proof}

\subsection{Proof of Theorem \ref{theo_3}}

\begin{proof} We start with a lemma that shows for an agent sequence, the infected set remains the same till agent 1 appears for the first time.
	\begin{lemma}\label{lem_6}
		Let $\uti\in N_\infty$ and $\hat{t}\in \mathbb{N}_0$ be such that $\uti_t\neq 1$ for all $t<\hat{t}$. Then, $I(S_t)=\{1\}$ for all $t\leqslant \hat{t}$.
	\end{lemma}
	\begin{proof}
		Note that if $\hat{t}=0$ then there is nothing to show. So, assume $\hat{t}\geqslant 1$. We use induction to prove this. As the base case, we show that $I(S_1)=\{1\}$. Let $\uti_0=i$. Since $\hat{t}\geqslant 1$, $i\neq 1$. Moreover, as $\tau <a <1$, $r_i(S_0)=\frac{1}{(n-1)}$. Hence, 
		$$b_i(S_0)=\min\left\{1,\frac{\tau}{\frac{1}{(n-1)}}\right\}=\min\{1,(n-1)\tau\}=1$$
		as by our assumption $\tau\geqslant \frac{1}{(n-1)}$. This means agent $i$ will not get infected. For any $j\notin \{1,i\}$, $a_j(\hat{S}_0)=a_j(S_0)=a$ and $r_j(\hat{S}_0)=\frac{a}{(n-2)a+1}\leqslant \frac{1}{(n-1)}$. Thus, $$a_j(\hat{S}_0)r_j(\hat{S}_0)=a\frac{a}{(n-2)a+1}\leqslant \frac{a}{(n-1)}\leqslant \frac{1}{(n-1)}\leqslant \tau.$$ So, agent $j$ will also not get infected at $t=1$. Thus, $I(S_1)=\{1\}$. Next we introduce an induction hypothesis.

			%
		
		\noindent\textit{Induction Hypothesis:} Given $\bar{t}\in \mathbb{N}_0$ with $\hat{t}\geqslant \bar{t}>1$, we have $I(S_1)=\cdots=I(S_{\bar{t}-1})=\{1\}$.
		
		We show that $I(S_{\bar{t}})=\{1\}$. Let $\uti_{\bar{t}-1}=i$. Since $\hat{t}\geqslant \bar{t}$, this means $i\neq 1$. Hence, $i\notin I(S_{\bt-1})$. As $\bt>1$, we have $I(S_{\bar{t}-2})=I(S_{\bt-1})$. This together with Lemma \ref{lem_2}, implies $I(S_{\bt})=I(S_{\bt-1})=\{1\}$. Thus, by induction, we have $I(S_{\hat{t}})=\{1\}$. This completes the proof of the lemma.  \end{proof}
	
	We complete the proof in two steps. In Step 1, we explore how the infection spreads when agents update their actions according to a fixed agent sequence, and in Step 2 we use this to explore how infection spreads when agents update their actions randomly.

	\noindent\textbf{Step 1:} Fix an agent sequence  $\uti\in N_\infty$ and let $S$ be the DVSP induced by $\undertilde{v}$. To shorten notation, for all $i \in N$, let us denote $t_i(\uti)$ by $k_i$. 
	
	\vspace{2mm}
	\noindent\textbf{Claim 1}: For all $0\leqslant t<k_1$, $a_i(S_{t+1})=1$ where $\uti_t=i$.
	
	\noindent \textbf{Proof of the claim.}  Let $\uti_0=i$. As $k_1>0$, $i\neq 1$. Since $a_j(S_0)=a>0$ for all $j\in N$ and $I(S_0)=\{1\}$, we have $r_i(S_0)=\frac{1}{(n-1)}$. This means $$b_i(S_0)=\min\left\{1,\frac{\tau}{\frac{1}{(n-1)}}\right\}=\min\{1,(n-1)\tau\}=1,$$ as by our assumption $\tau\geqslant \frac{1}{(n-1)}$. Thus, $a_i(S_1)=a_i(\hat{S}_0)=1$. Next we introduce an induction hypothesis.
	
	\noindent\textit{Induction Hypothesis:} Given $\bar{t}\in \mathbb{N}_0$ with $\bt<k_1$, we have for all $t<\bt$, $a_j(S_{t+1})=1$ where $\uti_t=j$.

	Let $\uti_{\bt}=i'$ and we show that $a_{i'}(S_{\bt+1})=1$. Note that by Lemma \ref{lem_6}, $I(S_{\bt})=\{1\}$. Moreover, by the induction hypothesis, $a_j(S_{\bt})\geqslant a$ for all $j\in N\setminus \{1\}$. Also, as $\bt<k_1$, we have $a_1(S_{\bt})=a$. Combining all these observations we have, 
	\begin{equation}\label{eq_2}
		\frac{1}{(n-1)}\geqslant r_{i'}(S_{\bt})\geqslant \frac{a}{(n-1)}.
	\end{equation}
	Since $r_{i'}(S_{\bt})>0$, $b_{i'}(S_{\bt})=\min\left\{1,\frac{\tau}{r_{i'}(S_{\bt})}\right\}$; see Remark \ref{rem_0}. Therefore, using (\ref{eq_2}) and the fact $\tau\geqslant \frac{1}{(n-1)}$, we have  $b_{i'}(S_{\bt})=1$. Thus, $a_{i'}(S_{\bt+1})=a_{i'}(\hat{S}_{\bt})=1$. This completes the proof of the claim. \hfill$\square$
	
	We distinguish some cases based on $|N_1(\uti)|$.
	
	\vspace{2mm}
	\noindent\textbf{Case 1:} $|N_1(\uti)|\geqslant \talpha$.\\
	We show that no new agent will get infected and $I(S_\infty)=\{1\}$. By Claim 1,  $a_i(S_{k_1})=1$ for all  $i\in N_1(\uti)$. By the definition of the process, $a_i(S_{k_1})=a$ for all  $i\notin N_1(\uti)\cup \{1\}$ as they have not updated their actions till the time point $k_1$. Recall that $\hat{S}_{k_1}$ denotes the intermediate state where the only change from $S_{k_1}$ is that agent $\uti_{k_1}$ has updated her action to $b_{\uti_{k_1}}(S_{k_1})$. Since $\uti_{k_1}=1$, we have  $a_i(S_{k_1})=a_i(\hat{S}_{k_1})$ for all $i\neq 1$. Thus, $a_i(\hat{S}_{k_1})=1$ for all $i\in |N_1(\uti)|$  and $a_i(\hat{S}_{k_1})=1$ for all $i\notin N_1(\uti)\cup \{1\}$.
	
	By Remark \ref{rem_0} and the definition of the process,  $a_1(\hat{S}_{k_1})=1$. Consider the time point $k_1+1$. By the definition of the process, an agent $i\neq 1$ will be in $I(S_{k_1+1})$ if $a_i(\hat{S}_{k_1})r_i(\hat{S}_{k_1})>\tau$. Since $I(S_{k_1})=\{1\}$, $a_i(\hat{S}_{k_1})=1$ for all $i\in N_1(\uti) \cup \{1\}$, and $a_i(\hat{S}_{k_1})=a$ for all $i\notin N_1(\uti)\cup \{1\}$, it follows that for all $i\in N_1(\uti)$ 
	$$r_i(\hat{S}_{k_1})= \frac{1}{|N_1(\uti)|+(n-1-|N_1(\uti)|)a}$$ and 
	\begin{align}
		a_i(\hat{S}_{k_1})r_i(\hat{S}_{k_1})&= \frac{1}{|N_1(\uti)|+(n-1-|N_1(\uti)|)a}\nonumber\\
		&=\frac{1}{|N_1(\uti)|(1-a)+(n-1)a}. \label{eq_3}
	\end{align}
	Recall that by the assumption of the case, $|N_1(\uti)|\geqslant \talpha$. This together with $\talpha=\max\left\{1, \left\lceil \frac{1-(n-1)a\tau}{\tau(1-a)}\right\rceil\right\}$ implies
	\begin{align}
		|N_1(\uti)|\geqslant  \frac{1-(n-1)a\tau}{\tau(1-a)}
		\implies \tau \geqslant \frac{1}{|N_1(\uti)|(1-a)+(n-1)a}. \label{eq_4}
	\end{align}
	Combining (\ref{eq_3}) and (\ref{eq_4}), we may conclude that agent $i$ will not be infected at the time point $t+1$. Similar arguments show that any agent $j\notin N\cup \{1\}$ will not be infected at the time point $t+1$. Hence, $I(S_{k_1+1})=\{1\}$. 
	
	We show that no new agent would get infected after this. We first show that $I(S_{k_1+2})=\{1\}$. Let $\uti_{k_1+1}=i$. If $i\notin I(S_{k_1+1})$ then as $I(S_{k_1})=I(S_{k_1+1})$ by Lemma \ref{lem_2}, we have $I(S_{k_1+1})=I(S_{k_1+2})$. If $i\in I(S_{k_1+1})$ then $i=1$. Moreover, $a_1(S_{k_1+1})=a_1(\hat{S}_{k_1})=1$. Hence, by Lemma \ref{lem_2}, $I(S_{k_1+1})=I(S_{k_1+2})$. Therefore, $I(S_{k_1+2})=\{1\}$. Using the same arguments repeatedly, it follows that $I(S_t)=\{1\}$ for all $t\geqslant k_1+2$. Thus,  $I(S_\infty)=\{1\}$.
	
	\vspace{2mm}
	\noindent\textbf{Case 2:} $|N_1(\uti)|\leqslant \talpha-1$.\\
In the following claim, we show that at time point $k_1+1$, the infected set is $N_1(\uti)\cup \{1\}$.

\begin{claim}
		$I(S_{k_1+1})=N_1(\uti)\cup \{1\}$.
\end{claim}
\noindent \textbf{Proof of the claim:} 
Recall that $\talpha=\max\left\{1, \left\lceil \frac{1-(n-1)a\tau}{\tau(1-a)}\right\rceil\right\}$. First assume $\talpha\neq \left\lceil \frac{1-(n-1)a\tau}{\tau(1-a)}\right\rceil$, i.e.,  $\talpha=1$ and $1-(n-1)a\tau\leqslant 0$. This, together with the assumption of the case, implies $|N_1(\uti)|=0$.  Therefore, $k_1=1$. Hence, to prove the claim, it is enough to show that $I(S_1)=\{1\}$. Note that by the definition of the process, $a_1(\hat{S}_0)=1$ and  $a_i(\hat{S}_{0})=a$ for all $i\neq 1$. Thus,

 	\begin{align*}
		r_i(\hat{S}_{0})&=\frac{1}{1+(n-2)a} \\
		&\leqslant\frac{1}{(n-1)a} \\
		&\leqslant \tau. \hspace{20mm}\text{ (since } 1-(n-1)a\tau\leqslant 0)\\
	\end{align*}
 This implies $I(S_{1})=\{1\}$. Now assume $\talpha=\left\lceil \frac{1-(n-1)a\tau}{\tau(1-a)}\right\rceil$. Consider an agent $i\in N_1(\uti)$. Using similar arguments as in (\ref{eq_3}), we may show that $$a_i(\hat{S}_{k_1})r_i(\hat{S}_{k_1})=\frac{1}{|N_1(\uti)|(1-a)+(n-1)a}.$$ This, together with  $\talpha=\left\lceil \frac{1-(n-1)a\tau}{\tau(1-a)}\right\rceil$ and $|N_1(\uti)|\leqslant \talpha-1$, implies  $a_i(\hat{S}_{k_1})r_i(\hat{S}_{k_1})>\tau$ and hence, agent $i$ will get infected at the time point $t+1$. For any $j\notin N_1(\uti)\cup \{1\}$,
	$$r_j(\hat{S}_{k_1})= \frac{1}{|N_1(\uti)|+1+(n-1-|N_1(\uti)|-1)a}$$ and 
	\begin{align}
		a_j(\hat{S}_{k_1})r_j(\hat{S}_{k_1})&= \frac{a}{|N_1(\uti)|+1+(n-1-|N_1(\uti)|-1)a}\nonumber\\
		&=\frac{a}{|N_1(\uti)|(1-a)+(n-2)a+1}. \nonumber
	\end{align}
	Hence, $j$ gets infected at $t+1$ if 
	\begin{align}
		\frac{a}{|N_1(\uti)|(1-a)+(n-2)a+1}>\tau 
		\implies  \frac{a-\tau(n-1)a}{\tau(1-a)}>|N_1(\uti)|. \nonumber
	\end{align}
	But this does not hold as $\frac{a-\tau(n-1)a}{\tau(1-a)}\leqslant 0$ and $|N_1(\uti)|\geqslant 0$. So, agent $j$ does not get infected at $t+1$. Thus, $I(S_{k_1+1})=N_1(\uti)\cup 1$. This completes the proof of the claim. \hfill$\square$
	
	We now determine the final infected set. To do so we consider two sub-cases based on the value of $|N_1(\uti)|$.
	
	\vspace{2mm}
	\noindent\textbf{Case 2.1:} $|N_1(\uti)|+1< \balpha$.\\
	We show that no new agent would get infected after $k_1+1$. We first show that $I(S_{k_1+2})=N_1(\uti)\cup \{1\}$. Let $\uti_{k_1+1}=i$. If $i\in I(S_{k_1+1})$ then $i\in N_1(\uti)\cup \{1\}$. Moreover, $a_i(S_{k_1+1})=a_i(\hat{S}_{k_1})=1$. Hence, by Lemma \ref{lem_2}, $I(S_{k_1+1})=I(S_{k_1+2})$. If $i\notin I(S_{k_1+1})$ then since $r_i(S_{k_1+1})=\frac{|N_1(\uti)|+1}{|N_1(\uti)|+1+(n-1-|N_1(\uti)|-1)a}\neq 0$, agent $i$ will choose $\min\left\{1,\frac{\tau}{r_i(S_{k_1+1})}\right\}$ as her action $a_i(\hat{S}_{k_1+1})$ at $\hat{S}_{k_1+1}$. This means $a_i(\hat{S}_{k_1+1})r_i(\hat{S}_{k_1+1})\leqslant \tau$ and agent $i$ will not get infected at $k_1+2$. To show that any agent $j\in I(S_{k_1+1})\setminus \{i\}$ will not get infected at $k_1+2$, we first prove a claim. 
	
	\begin{claim}\label{cl_1}
		$a_i(\hat{S}_{k_1+1})\geqslant a$.
	\end{claim}
	\noindent \textbf{Proof of the claim:} Note that if	$a_i(\hat{S}_{k_1+1})=1$ then the claim holds as $a\leqslant 1$. If $a_i(\hat{S}_{k_1+1})=\frac{\tau}{r_i(S_{k_1+1})}$ then 
	\begin{align}
		a_i(\hat{S}_{k_1+1})&=\frac{\tau}{r_i(S_{k_1+1})} \nonumber\\
		&=\tau(1-a)+\frac{\tau a(n-1)}{|N_1(\uti)|+1}. \label{eq_5}
	\end{align}
	Moreover, by the assumption of the case $|N_1(\uti)|+1<\balpha$. This together with $\balpha =\left\lfloor \frac{(n-1)a\tau}{a-\tau(1-a)}\right\rfloor +1$ and (\ref{eq_5}) implies
	\begin{align}
		a_i(\hat{S}_{k_1+1})&\geqslant \tau(1-a)+\frac{[\tau a(n-1)](a-\tau(1-a))}{(n-1)a\tau}. \nonumber \\
		&=a. \nonumber
	\end{align}
	This completes the proof of the claim. \hfill$\square$
	
	For any $j\in I(S_{k_1+1})\setminus \{i\}$, $a_j(\hat{S}_{k_1+1})=a$ and $$r_j(\hat{S}_{k_1+1})=\frac{|N_1(\uti)|+1}{|N_1(\uti)|+1+(n-1-|N_1(\uti)|-2)a+a_i(\hat{S}_{k_1+1})}.$$ Thus,
	\begin{align}
		a_j(\hat{S}_{k_1+1})r_j(\hat{S}_{k_1+1})&=\frac{a(|N_1(\uti)|+1)}{|N_1(\uti)|+1+(n-1-|N_1(\uti)|-2)a+a_i(\hat{S}_{k_1+1})} \nonumber \\
		&\leqslant \frac{a(|N_1(\uti)|+1)}{|N_1(\uti)|+1+(n-1-|N_1(\uti)|-1)a}  \hspace{10mm} \mbox{ (as by Claim } \ref{cl_1}, a_i(\hat{S}_{k_1+1})\geqslant a) \nonumber \\
		&=ar_i(\hat{S}_{k_1+1}) \nonumber \\
		&\leqslant \tau \hspace{10mm} \mbox{ (as } a_i(\hat{S}_{k_1+1})r_i(\hat{S}_{k_1+1})\leqslant \tau \text{ and } a_i(\hat{S}_{k_1+1})\geqslant a). \nonumber 
	\end{align}
	Hence, agent $j$ will not get infected at $k_1+2$. This concludes that $I(S_{k_2+1})=N_1(\uti)\cup \{1\}$. Now using similar logic as in Case 1, we may show that no agent would get infected after this and $I(S_\infty)=N_1(\uti)\cup \{1\}$. 
	
	\vspace{2mm}
	\noindent\textbf{Case 2.2:} $|N_1(\uti)|+1\geqslant \balpha$.\\
	First assume that $\uti_{k_1+1}=i$ where $i\in N_1(\uti)\cup \{1\}$. We show that $I(S_{\infty})=N$. Note that as $i\in N_1(\uti)\cup \{1\}$, $a_i(\hat{S}_{k_1+1})=1$. Thus, for any $j\notin N_1(\uti)\cup \{1\}$,
	$$r_j(\hat{S}_{k_1+1})=\frac{|N_1(\uti)|+1}{|N_1(\uti)|+1+(n-1-|N_1(\uti)|-1)a}$$ and hence, 
	\begin{align}
		a_j(\hat{S}_{k_1+1})r_j(\hat{S}_{k_1+1})&= \frac{a(|N_1(\uti)|+1)}{|N_1(\uti)|+1+(n-1-|N_1(\uti)|-1)a}   \nonumber \\
		&= \frac{a(|N_1(\uti)|+1)}{(|N_1(\uti)|+1)(1-a)+(n-1)a}\nonumber \\
		&\geqslant  \frac{a\balpha}{\balpha(1-a)+(n-1)a}  \hspace{10mm} \mbox{ (as } |N_1(\uti)|+1\geqslant \balpha) \nonumber \\
		&> \tau.  \hspace{35mm} \mbox{ (as } \balpha>\frac{(n-1)a\tau}{a-\tau(1-a)}) \nonumber
	\end{align}
	
	Therefore, $I(S_{k_1+2})=N$ and  $I(S_\infty)=N$. 
	
	Now assume that $\uti_{k_1+1}=i$ where $i\notin N_1(\uti)\cup \{1\}$. We show that $I(S_{k_1+2})=N\setminus i$.  
	%
	%
	%
	Since $i\notin I(S_{k_1+1})$ and $\uti_{k_1+1}=i$, agent $i$ will not get infected at $k_1+2$ (Observation \ref{obs_4}). Consider $j\notin N_1(\uti)\cup 1$ with $j\neq i$.  We first prove a claim.
	
	\begin{claim}\label{cl_2}
		$\tau<a_i(\hat{S}_{k_1+1})< a$.
	\end{claim}
	\noindent \textbf{Proof of the claim:}  We show that $\tau<\frac{\tau}{r_i(S_{k_1+1})}<a$. This together with $a<1$ proves the claim. As $r_i(S_{k_1+1})=\frac{|N_1(\uti)|+1}{|N_1(\uti)|+1+(n-1-|N_1(\uti)|-1)a}<1$, we have $\tau<\frac{\tau}{r_i(S_{k_1+1})}$. To see $\frac{\tau}{r_i(S_{k_1+1})}<a$,  recall that by (\ref{eq_5}) 
	$$a_i(\hat{S}_{k_1+1})=\tau(1-a)+\frac{\tau a(n-1)}{|N_1(\uti)|+1}.$$
	
	\noindent Moreover, by the assumption of the case $|N_1(\uti)|+1\geqslant \balpha$. This together with $\balpha > \frac{(n-1)a\tau}{a-\tau(1-a)}$ implies
	\begin{align}
		a_i(\hat{S}_{k_1+1})&< \tau(1-a)+\frac{[\tau a(n-1)](a-\tau(1-a))}{(n-1)a\tau}. \nonumber \\
		&=a. \nonumber
	\end{align}
	This completes the proof of the claim. \hfill$\square$
	
	For any $j\in I(S_{k_1+1})\setminus \{i\}$, $a_j(\hat{S}_{k_1+1})=a$ and $$r_j(\hat{S}_{k_1+1})=\frac{|N_1(\uti)|+1}{|N_1(\uti)|+1+(n-1-|N_1(\uti)|-2)a+a_i(\hat{S}_{k_1+1})}.$$ Thus,
	\begin{align}
		a_j(\hat{S}_{k_1+1})r_j(\hat{S}_{k_1+1})&=\frac{a(|N_1(\uti)|+1)}{|N_1(\uti)|+1+(n-1-|N_1(\uti)|-2)a+a_i(\hat{S}_{k_1+1})} \nonumber \\
		&> \frac{a(|N_1(\uti)|+1)}{|N_1(\uti)|+1+(n-1-|N_1(\uti)|-1)a}  \hspace{10mm} \mbox{ (as by Claim } \ref{cl_2}, a_i(\hat{S}_{k_1+1})< a) \nonumber \\
		&= \frac{a(|N_1(\uti)|+1)}{(|N_1(\uti)|+1)(1-a)+(n-1)a}\nonumber \\
		&\geqslant  \frac{a\balpha}{\balpha(1-a)+(n-1)a}  \hspace{10mm} \mbox{ (as } |N_1(\uti)|+1\geqslant \balpha) \nonumber \\
		&> \tau.  \hspace{10mm} \mbox{ (as } \balpha>\frac{(n-1)a\tau}{a-\tau(1-a)}) \nonumber
	\end{align}
	Hence, agent $j$ will get infected at $k_1+2$. This concludes that $I(S_{k_1+2})=N\setminus \{i\}$. 
	
	To determine the final infected set, we now distinguish two cases based on whether $\uti_{k_1+2}=i$ or not.
	
	\vspace{2mm}
	\noindent\textbf{Case 2.2.1:}  $\uti_{k_1+2}=i$ \\
	We show that the final infected set will be $N\setminus i$. Since by our assumption $\uti_{k_1+2}=i$ and $i\notin I(S_{k_1+2})$, by Observation \ref{obs_4}, $i\notin I(S_{k_1+3})$. Hence, $I(S_{k_1+3})=N\setminus \{i\}$. We now show that $i$ will not get infected after this. At time point $k_1+2$,
	$$r_i(\hat{S}_{k_1+2})=\frac{(|N_1(\uti)|+1)+a(n-1-|N_1(\uti)|-1)}{(|N_1(\uti)|+1)+(n-1-|N_1(\uti)|-1)a}=1.$$ Therefore, $a_i(\hat{S}_{k_1+2})=\tau$. At $k_1+3$, if $\uti_{k_1+3}=i$, then agent $i$ would not get infected at $k_1+4$ (Observation \ref{obs_4}). On the other hand, if $\uti_{k_1+3}\neq i$ then as $a_i(\hat{S}_{k_1+3})=a_i(\hat{S}_{k_1+2})=\tau$, agent $i$ would remain uninfected at $k_1+4$. Continuing in this manner, we may show that $i$ will  not get infected after this. Thus, $I(S_\infty)=N\setminus \{i\}$.
	
	\vspace{2mm}
	\noindent\textbf{Case 2.2.2:} $\uti_{k_1+2}\neq i$ \\
	We show that the final infected set will be $N$. Since  $I(S_{k_1+2})=N\setminus \{i\}$, $r_i(\hat{S}_{k_1+2})=1$. Moreover as $a_i(S_{k_1+2})=a_i(\hat{S}_{k_1+1})>\tau$ (by Claim \ref{cl_2}) and $\uti_{k_1+2}\neq i$, it follows that $a_i(\hat{S}_{k_1+2})>\tau$. Combining this two we have $a_i(\hat{S}_{k_1+2})r_i(\hat{S}_{k_1+2})>\tau$. Thus, agent $i$ will get infected at $k_1+3$. Hence, $I(S_{k_1+3})=N$ and  $I(S_\infty)=N$. 
	
	\noindent\textbf{Step 2:} We now find the distribution of $I(S_\infty)$. Note that by the definition, $\talpha\leqslant n-1$ and $\balpha\geqslant 2$. First assume that $\talpha+1\leqslant \balpha$. Therefore, by the above cases we have  
	\begin{itemize}
		\item $I(S_\infty)=\{1\}$ if $|N_1(\uti)|\in \{0,\talpha,\talpha+1,\ldots,n-1\}$,
		\item $I(S_\infty)=N_1(\uti)\cup \{1\}$ if $|N_1(\uti)|\in \{1,2,\ldots,\talpha-1\}$.
	\end{itemize}
	Moreover, as $\mathbb{P}$ is uniform, any two subsets of $N$ with same cardinality have the same probability. These observations together  with  Lemma \ref{lem_5} yield 
	
\[
		\mathbb{P}(I(S_\infty)=J)= 
		\begin{cases} 
			1-\frac{\tilde{\alpha}-1}{n} & \text{if } J=\{1\}  , \\ \\
			\frac{1}{n \times \Mycomb[n-1]{m-1}} & \text{if } 
			1\in J \text{ and } |J|=m \text{ where } m\in [2,\talpha], \\ \\
			0 &  \text{ otherwise. }
		\end{cases}
		\]
	
	Now assume that $\talpha+1> \balpha\geqslant 2$. By Case 1 and Case 2, we have
	\begin{enumerate}[(i)]
		\item $I(S_\infty)=\{1\}$ if $|N_1(\uti)|\in \{0,\talpha,\talpha+1,\ldots,n-1\},$
		\item $|I(S_\infty)|=|N_1(\uti)|+1$ with $1\in I(S_\infty)$ if $|N_1(\uti)|\in \{1,2,\ldots,\balpha-2\},$
		\item $|I(S_\infty)|=n$ if $|N_1(\uti)|\in \{\balpha-1,\ldots,\talpha-1\}$ and there is no $i\in N$ such that $k_i=k_1+1$ and $\uti_{k_1+2}=i$, and
		\item $|I(S_\infty)|=n-1$ with $1\in I(S_\infty)$ if $|N_1(\uti)|\in \{\balpha-1,\ldots,\talpha-1\}$ and there is $i\in N$ such that $k_i=k_1+1$ and $\uti_{k_1+2}=i$.
	\end{enumerate}
	
	Since $|N_1|$ follows uniform distribution on $\{0,1,\ldots,n-1\}$ and any two subsets of $N$ with the same cardinality have the same probability, by (i) and (ii) we have 
	
		\[
	\mathbb{P}(I(S_\infty)=J)=
	   \begin{cases}
	    1-\frac{\talpha-1}{n} & \text{if }  J=\{1\},\\ \\
\frac{1}{n \times \Mycomb[n-1]{m-1}} & \text{if } 1\in J \text{ and } |J|=m \text{ where } m\in[2,\balpha-1].	   \end{cases}
	  \]


	We calculate the probability of $|I(S_\infty)|=n-1$. By (iv) we have
	\begin{align*}
		&\mathbb{P}(\uti\mid |N_1(\uti)|\in \{\balpha-1,\ldots,\talpha\} \text{ and } \exists i\neq 1 \text{ such that } k_i=k_1+1 \text{ and } \uti_{k_1+2}=i)\\
		=&\sum_{w=\balpha-1}^{\talpha-1}P(\uti\mid |N_1(\uti)|=w \text{ and } \exists i\neq 1 \text{ such that } k_i=k_1+1 \text{ and } \uti_{k_1+2}=i)\\
		=&\sum_{w=\balpha-1}^{\talpha-1}\sum_{t= w+1}^{\infty}P(\uti\mid |N_1(\uti)|=w \text{ and } k_1=t \text{ and } \exists i\neq 1 \text{ such that } k_i=t+1 \text{ and } \uti_{t+2}=i)\\
		=&\sum_{w=\balpha-1}^{\talpha-1}\sum_{t= w+1}^{\infty} \Mycomb[n-1]{w}\times \left(\frac{w!\stirling{t-1}{w}}{n^{t-1}}\right)\times \frac{1}{n} \times \Mycomb[(n-w-1)]{1}\times \frac{1}{n^2}\\
		=&\frac{(n-1)!}{n^3}\sum_{w=\balpha-1}^{\talpha-1}\frac{1}{(n-w-2)!}\sum_{t= w+1}^{\infty}  \left(\frac{\stirling{t-1}{w}}{n^{t-1}}\right)\\
		=&\eta(\talpha,\balpha,n).        
	\end{align*}
	Note that by (i)-(iv), $$\sum_{m=1}^{\balpha-1}P(|I(S_\infty)|=m)+P(|I(S_\infty)|=n-1)+P(|I(S_\infty)|=n)=1.$$ Therefore,
	\begin{align*}
		P(|I(S_\infty)|=n)&=1-\sum_{m=1}^{\balpha-1}P(|I(S_\infty)|=m)-P(|I(S_\infty)|=n-1)\\
		&=\frac{\talpha-(\balpha-1)}{n}-\eta(\talpha,\balpha,n). 
	\end{align*}
Since any two subsets of $N$ with the same cardinality have the same probability,	combining all the above observations, we have the following distribution of the infected set.

	\[
	\mathbb{P}(I(S_\infty)=J)= 
	\begin{cases} 
		1-\frac{\tilde{\alpha}-1}{n} & \text{if } J=\{1\}  , \\ \\
		\frac{1}{n \times \Mycomb[n-1]{m-1}} & \text{if } 
		1\in J \text{ and } |J|=m \text{ where } m\in[2, \balpha-1], \\ \\
		\frac{\eta(\talpha,\balpha,n)}{n-1} & \text{if } 1\in J \text{ and } |J|=n-1,\\ \\
		\frac{\talpha-(\balpha-1)}{n}-\eta(\talpha,\balpha,n) & \text{if } |J|=n, \text{ i.e., } J=N,\\ \\
		0 &  \text{ otherwise.}
	\end{cases}
	\]


This completes the proof of the theorem.
\end{proof}

\subsection{Proof of Theorem \ref{theo_4}}

\begin{proof} 
	We follow the same structure that we used in the proof of Theorem \ref{theo_1}. 
	
	\noindent \textbf{Step 1.}    Fix an agent sequence  $\uti\in N_\infty$ and let $S$ be the DVSP induced by $\undertilde{v}$. To shorten notation, for all $i \in N$, let us denote $t_i(\uti)$ by $k_i$.  The following  claim demonstrates how an agent $i$ with $k_i<k_1$ will update her action. Recall the set $N_1(\uti)$. We distinguish two cases based on the value of $|N_1(\uti)|$. \\

	\noindent\textbf{Case 1:} $|N_1(\uti)|=0$. \\
	We show that, for $\tau\leqslant \frac{a}{n-1}$, all the agents will get infected under this assumption, i.e., $I(S_\infty)=N$. Note that by the assumption of the case, $\uti_0=1$.   Recall that $\hat{S}_{0}$ denotes the intermediate state where the only change from $S_{0}$ is that agent $\uti_{0}$ has updated her action to $b_{\uti_{0}}(S_{0})$. Since $\uti_{0}=1$, we have  $a_i(S_{0})=a_i(\hat{S}_{0})=a$ for all $i\neq 1$. Moreover, by Remark \ref{rem_0} and the definition of the process,  $a_1(\hat{S}_{0})=1$. Consider the time point $1$. By the definition of the process, an agent $i\neq 1$ will be in $I(S_{1})$ if $a_i(\hat{S}_{0})r_i(\hat{S}_{0})>\tau$. Since $I(S_{0})=\{1\}$, $a_i(\hat{S}_{0})=a$ for all $i\in N$, it follows that for all $i\in N\setminus \{1\}$
	\begin{align*}
		ar_i(\hat{S}_{0})&= \frac{a}{(n-2)a+1}> \frac{a}{(n-1)}.
	\end{align*}
	  Because $\tau\leqslant  \frac{a}{n-1}$, this implies that all the agents in $N\setminus \{1\}$  gets infected at the time point $1$. Hence, $I(S_{1})=N$. Therefore, by the definition of the process $I(S_\infty)=N$.\\

	\noindent\textbf{Case 2:} $|N_1(\uti)|\geqslant  1$.\\ 
	This means $\uti_0\neq 1$. Let $\uti_0=i\in N\setminus 1$. Hence, by the definition of the process, agent $i$ will choose her action as $b_i(S_0)$ at the intermediate state $\hat{S}_0$. As $a_j(S_0)=a> 0$ for all $j\in N$ and $I(S_0)=\{1\}$, it follows that $r_i(S_0)\neq 0$. Therefore,
	\begin{equation}\label{e_1}
		b_i(S_0)=\min \left\{1,\frac{\tau}{r_i(S_0)}\right \}=\min \left \{1,(n-1)\tau\right \}=(n-1)\tau.    
	\end{equation}

	Since by our assumption $\uti_{0}=i$ and $i\notin I(S_{0})$, by Observation \ref{obs_4}, $i\notin I(S_{1})$. For any other uninfected agent $j$, $$r_j(\hat{S}_0)=\frac{a}{(n-2)a+b_i(S_0)}=\frac{a}{(n-2)a+(n-1)\tau}.$$ This together with the fact that $a_j(\hat{S}_0)=a$ implies 
	\begin{enumerate}
		\item if $\tau= \frac{a}{n-1}$ then  $a_j(\hat{S}_0)r_j(\hat{S}_0)=\frac{a}{n-1}= \tau$, and
		\item if $\tau< \frac{a}{n-1}$ then  $a_j(\hat{S}_0)r_j(\hat{S}_0)>\frac{a}{n-1}>\tau$.
	\end{enumerate}
	Combining the above observations, we may write if $\tau= \frac{1}{n-1}$ then agent $j$ will not get infected at time point $1$ and if $\tau< \frac{a}{n-1}$ then agent $j$ will get infected at time point $1$.  Hence, we have
	$$\tau= \frac{a}{n-1} \implies I(S_1)=\{1\} \mbox{ and } \tau< \frac{a}{n-1} \implies I(S_1)=N\setminus \{i\}.$$
	To decide the final outcome, we first assume $\tau= \frac{a}{n-1}$. Note that by (\ref{e_1}), $b_i(S_0)=a$. This means $a_i(S_1)=a$. Moreover, as $\uti_0=i$, we have $a_j(S_1)=a$ for all $j\neq i$. Using similar arguments, we can show that $a_k(S_{k_1})=a$ for all $k\in N$ and $I(S_{k_1})=\{1\}$. By the definition of the process, $a_1(\hat{S}_{k_1})=1$ and $a_k(\hat{S}_{k_1})=a$ for all $k\neq 1$. Therefore, for any $k\neq 1$
	 \begin{align*}
	 	 a_k(\hat{S}_{k_1})r_k(\hat{S}_{k_1})&=\frac{a}{(n-2)a+1} >\frac{a}{(n-1)}=\tau.
	\end{align*}
	 Thus, all the agents other than agent 1 will get infected at $k_1+1$. Hence, $I(S_\infty)=N$.
	 
	Now assume $\tau<\frac{a}{n-1}$. We distinguish two subcases.

	\noindent \textbf{Case 2.1.}  $\uti_1=i$.\\
	We show that the final infected set will be $N\setminus i$. Since by our assumption $\uti_{1}=i$ and $i\notin I(S_{1})$, by Observation \ref{obs_4}, $i\notin I(S_{2})$. Hence, $I(S_{2})=N\setminus \{i\}$. We now show that $i$ will not get infected after this. At time point $2$,
	$$r_i(\hat{S}_{2})=\frac{(n-1)}{(n-1)}=1.$$ Therefore, $a_i(\hat{S}_{2})=\tau$ (see Observation \ref{obs_4}). At time point $3$, if $\uti_{3}=i$, then agent $i$ would not get infected at time point $4$ (Observation \ref{obs_4}). On the other hand, if $\uti_{3}\neq i$ then as $a_i(\hat{S}_{3})=a_i(\hat{S}_{2})=\tau$, it follows that $a_i(\hat{S}_{3})r_i(\hat{S}_3)\leqslant \tau$. Hence, agent $i$ would remain uninfected at time point $4$. Continuing in this manner, we may show that $i$ will  not get infected after this. Thus, $I(S_\infty)=N\setminus \{i\}$.

	\vspace{2mm}
	\noindent\textbf{Case 2.2.:} $\uti_{1}\neq i$ \\
	We show that the final infected set will be $N$. Since  $I(S_{1})=N\setminus \{i\}$, $r_i(\hat{S}_{1})=1$. Moreover, as $a_i(S_{1})=a_i(\hat{S}_{0})=b_i(S_0)=(n-1)\tau>\tau$ (see \ref{e_1}) and $\uti_{1}\neq i$, it follows that $a_i(\hat{S}_{1})>\tau$. Combining this two we have $a_i(\hat{S}_{1})r_i(\hat{S}_{1})>\tau$. Thus, agent $i$ will get infected at time point $2$. Hence, $I(S_{2})=N$ and  $I(S_\infty)=N$.

	\noindent \textbf{Step 2.} First assume $\tau= \frac{a}{n-1}$. Therefore, in view Case 1 and Case 2 of the current proof, we have $I(S_\infty)=N$.

	Now assume $\tau<\frac{a}{n-1}$. By Case 1 and Case 2 above, we have 
	\begin{enumerate}[(i)]
		\item $I(S_\infty)=N\setminus i$ with $1\in I(S_\infty)$  if $|N_1(\uti)|\geqslant 1$ and there is $i\in N\setminus \{1\}$ such that $k_i=0$ and $\uti_{1}=i$, and
		\item $I(S_\infty)=N$  if either $|N_1(\uti)|=0$ or $|N_1(\uti)|\geqslant 1$ and there is no $i\in N\setminus \{1\}$ such that $k_i=0$ and $\uti_{1}=i$. 
	\end{enumerate}
	
	We calculate the probability of $|I(S_\infty)|=n-1$. By (i) we have
	\begin{align*}
		&P(\uti\mid |N_1(\uti)|\geqslant 1 \text{ and } \exists i\neq 1 \text{ such that } k_i=0 \text{ and } \uti_{1}=i)\\
		=&P(\uti\mid  \exists i\neq 1 \text{ such that } k_i=0 \text{ and } \uti_{1}=i)\\
		=&\Mycomb[n-1]{1}\times \frac{1}{n^2}\\
		=&\frac{n-1}{n^2}.
	\end{align*}
	Note that by (i) and (ii), 
	$$P(|I(S_\infty)|=n-1)+P(I(S_\infty)=N)=1.$$ Therefore,
	\begin{align*}
		P(I(S_\infty)=N)&=1-P(|I(S_\infty)|=n-1)\\
		&=1-\frac{n-1}{n^2}. 
	\end{align*}
	Since any two subsets of $N$ with the cardinality $n-1$ have the same probability,	combining all the above observations, we have the following distribution of the infected set. 
	
	\[
	\mathbb{P}(I(S_\infty)=J)= 
	\begin{cases} 
		\frac{1}{n^2} & \text{if } 1\in J \text{ and } |J|=n-1,
		\\ \\
		1-\frac{n-1}{n^2} & \text{if } |J|=n, \text{ i.e., } J=N, \\ \\
		0 & \text{ otherwise. } 
	\end{cases}
	\]    
	This completes the proof of the theorem. 
\end{proof}

\section{A few important lemmas}\label{sec:appendix_D}

\begin{lemma}\label{ac_1}
	Let $\uti\in N_\infty$ and let $S$ be the DVSP induced by $\uti$. Suppose $t_0$ is such that $I(S_{t_0})=I(S_t)$ for all $t\geqslant t_0$ and $a_k(S_{t_0})=1$ for all $k\in I(S_{t_0})$. Then for $i\notin I(S_{t_0})$ and $\bt>t_0$   with $\uti_{\bt}=i$ implies $$a_i(\hat{S}_{\bt})\geqslant a_j(\hat{S}_{\bt}) \text{ for all } j\notin I(S_{t_0}) \mbox{ with } \uti_t=j \text{ for some } t\in (t_0,\bt].$$
\end{lemma}

\begin{proof}
	We use induction on $\bt$ to prove the lemma. Note that for the base case, that is, for $\bt=t_0+1$, the lemma holds vacuously. Next we introduce an introduction hypothesis.\\
	\noindent\textit{Induction Hypothesis:} Given $\bt\in \mathbb{N}_0$ with $\bt>t_0+1$, the lemma holds for all $t$ with $t_0+1\leqslant t<\bt$.
	
	  We show that the lemma holds for $\bt$. Suppose $\uti_{\bt}=i$ where $i\notin I(S_{t_0})$. If there is no $t\in (t_0,\bt)$ such that  $\uti_t\notin I(S_{t_0})$, the lemma holds vacuously. So, assume that $\hat{t}$ is the last time point before $\bt$ such that $\uti_{\hat{t}}=j$ for some $j\notin I(S_{t_0})$. This, together with the induction hypothesis, implies $a_j(\hat{S}_{\hat{t}})\geqslant a_k(\hat{S}_{\hat{t}})$ for all $k\notin I(S_0)$ with $\uti_t=k$  for some  $t\in (t_0,\hat{t})$. Also, by the definition of the process, $a_l(\hat{S}_{\hat{t}})=a_l(\hat{S}_{\bt})$ for all $l\notin I(S_0)\setminus i$. Therefore, to prove the lemma it is enough to show that $a_i(\hat{S}_{\bt})\geqslant a_j(\hat{S}_{\bt})$. Additionally, as $a_j(\hat{S}_{\bt})\leqslant 1$, we may assume that $a_i(\hat{S}_{\bt})=\frac{\tau}{r_i(\hat{S}_{\bt})}$. Moreover, as $j\notin I(S_0)$, $a_j(\hat{S}_{\hat{t}})\leqslant\frac{\tau}{r_j(\hat{S}_{\hat{t}})}$. Now
	\begin{align}
		\frac{\tau}{r_j(\hat{S}_{\hat{t}})}&=\frac{\tau}{\frac{|I(S_{t_0})|}{|I(S_{t_0})|+\sum_{k\notin I(S_{t_0})\setminus j}a_k(\hat{S}_{\hat{t}})}}\nonumber\\
		&=\frac{\tau[{|I(S_{t_0})|+\sum_{k\notin I(S_{t_0})\setminus j}a_k(\hat{S}_{\hat{t}})}]}{|I(S_{t_0})|} \text{ (as } I(S_{t_0})=I(S_{\hat{t}}) \text{ and for } k\in I(S_{t_0}),\; a_k(S_{\hat{t}})=1 )\nonumber\\
		&=\frac{\tau[{|I(S_{t_0})|+\sum_{k\notin I(S_{t_0})\setminus \{i,j\}}a_k(\hat{S}_{\hat{t}})+a_i(\hat{S}_{\hat{t}})}]}{|I(S_{t_0})|} \nonumber\\
		&\leqslant \frac{\tau[{|I(S_{t_0})|+\sum_{k\notin I(S_{t_0})\setminus \{i,j\}}a_k(\hat{S}_{\hat{t}})+a_j(\hat{S}_{\hat{t}})}]}{|I(S_{t_0})|} \text{ (as } a_j(\hat{S}_{\hat{t}})\geqslant a_i(\hat{S}_{\hat{t}})) \nonumber\\
		&=\frac{\tau[{|I(S_{t_0})|+\sum_{k\notin I(S_{t_0})\setminus \{i,j\}}a_k(\hat{S}_{\bt})+a_j(\hat{S}_{\bt})}]}{|I(S_{t_0})|}  \text{ (as } a_k(\hat{S}_{\hat{t}})=a_k(\hat{S}_{\bt}) \text{ for all } k\notin I(S_0)\setminus i)\nonumber\\
		&=\frac{\tau[{|I(S_{t_0})|+\sum_{k\notin I(S_{t_0})\setminus i}a_k(\hat{S}_{\bt})}]}{|I(S_{t_0})|} \nonumber\\
		&=\frac{\tau}{\frac{|I(S_{t_0})|}{|I(S_{t_0})|+\sum_{k\notin I(S_{t_0})\setminus i}a_k(\hat{S}_{\bt})}} \text{ (as } I(S_{t_0})=I(S_{\bt}) \text{ and for } k\in I(S_{t_0}),\; a_k(S_{\bt})=1 ) \nonumber\\
		&=\frac{\tau}{r_i(\hat{S}_{\bt})}.\label{1}
	\end{align}
(\ref{1}) together with $a_i(\hat{S}_{\bt})=\frac{\tau}{r_i(\hat{S}_{\bt})}$ and $a_j(\hat{S}_{\hat{t}})\leqslant \frac{\tau}{r_j(\hat{S}_{\hat{t}})}$ implies $a_i(\hat{S}_{\bt})\geqslant a_j(\hat{S}_{\hat{t}})$. Hence, $a_i(\hat{S}_{\bt})\geqslant a_j(\hat{S}_{\bt})$.   This completes the proof of the lemma.
\end{proof}

The following lemma provides an important property of the final action limit for both infected and uninfected agents. It shows that an infected agent will have the action limit 1 whereas any two uninfected agents will have the same action limit, that is, for $i,j\notin I(S_\infty)$, $a_i(S_\infty)=a_j(S_\infty)$.  


\begin{lemma}\label{ac_2}
    Let $\uti\in N_\infty$  and let $S$ be the DVSP induced by $\uti$. Then, for $$[k\in I(S_\infty)] \implies [a_k(S_\infty)=1]$$
    and $$[i,j\notin I(S_\infty)]\implies [a_i(S_\infty)=a_j(S_\infty)].$$
\end{lemma}

\begin{proof}
Let $\uti\in N_\infty$  and let $S$ be the DVSP induced by $\uti$. Consider $k\in I(S_\infty)$. As $\uti \in N_\infty$, agent $k$ appears infinitely many times in $\uti$. And, after getting infected whenever she updates her action, she will choose her action as 1. Thus, $a_k(S_\infty)=1$.
Now consider $i,j \notin I(S_\infty)$. Let $b=a_i(S_\infty)$ and consider $\epsilon>0$. This means there exists $t_0$ such that $a_i(S_t)\geqslant b-\epsilon$ for all $t\geqslant t_0$. Note that as $N$ is a finite set and $I(S_\infty)$ exists, there exists $\tilde{t}\in \mathbb{N}_0$ such that $I(S_{\tilde{t}})=I(S_\infty)$. In view of this, we may assume that $I(S_{t_0})=I(S_\infty)$. Consider a time point $\hat{t}$ such that 
\begin{enumerate}
    \item $\hat{t}> t_0$ and $\uti_{\hat{t}}=j$ and
    \item there exists $\bar{t}\in [t_o,\hat{t}]$ such that $\uti_{\bt}=i$.
\end{enumerate}
 Such a time point $\hat{t}$ exists as $\uti \in N_\infty$. Therefore, by Lemma \ref{ac_1}, $a_j(S_{\hat{t}})\geqslant a_i(S_{\hat{t}})$. As $\hat{t}> t_0$, this means  $a_j(S_{\hat{t}})\geqslant b-\epsilon$. Further, as $I(S_{t_0})=I(S_\infty)$ and $\hat{t}> t_0$, by Claim \ref{cl_3} in Lemma \ref{lem_1}, $a_j(S_t)\geqslant a_j(S_{\hat{t}})$ for all $t\geqslant \hat{t}$. Thus, $a_j(S_t)\geqslant b-\epsilon$ for all $t\geqslant \hat{t}$. Since $\epsilon$ is arbitrary, this gives $a_j(S_{\infty})\geqslant b$. Similarly, we can show that $a_i(S_{\infty})\geqslant a_j(S_{\infty})$. Hence, $a_i(S_{\infty})= a_j(S_{\infty})$. 
\end{proof}

The next lemma determines the common action limit of the uninfected agents. 

 \begin{lemma}\label{2}
 Let $\uti\in N_\infty$  and let $S$ be the DVSP induced by $\uti$. Further, let $\gamma$ be the common action limit of the uninfected agents. Then, $$[(n-1)\tau<|I(S_\infty)|] \implies \left[\gamma=\frac{\tau|I(S_\infty)|}{(1+\tau)|I(S_\infty)|-\tau(n-1)}<1\right],$$ and $$[(n-1)\tau\geqslant |I(S_\infty)|] \implies [\gamma=1].$$
\end{lemma} 
 \begin{proof}
	Let $t_0\in \mathbb{N}_0$ be such that $I(S_{t_0})=I(S_\infty)$ and $a_k(S_{t_0})=1$ for all $k\in I(S_{t_0})$. First assume that $(n-1)\tau<|I(S_\infty)|$. This implies $\frac{\tau}{|I(S_\infty)|}<\frac{1}{n-1}$.  We first show that for any time point $\bt\geqslant t_0$, if $\uti_{\bt}\notin I(S_\infty)$ then $a_{\uti_{\bt}}(\hat{S}_{\bt})<1$. Let $\uti_{\bt}=i$. Since $a_{i}(\hat{S}_{\bt})=\min\{\frac{\tau}{r_i(\hat{S}_{\bt})},1\}$, it is enough to show that $\frac{\tau}{r_i(\hat{S}_{\bt})}<1$.
	\begin{align*}
		\frac{\tau}{r_i(\hat{S}_{\bt})}&=\frac{\tau}{|I(\hat{S}_{\bt})|}  \left(|I(\hat{S}_{\bt})|+\sum_{j\notin I(\hat{S}_{\bt})\cup \{i\}}a_j(\hat{S}_{\bt})\right)\\
		&=\frac{\tau}{|I(S_{\infty})|}  \left(|I(S_{\infty})|+\sum_{j\notin I(S_{\infty})\cup \{i\}}a_j(\hat{S}_{\bt})\right) \text{ (as } I(S_{t_0})=I(S_\infty) \text{ and } \bt\geqslant t_0)\\
		&<\frac{1}{n-1}  \left(|I(S_{\infty})|+\sum_{j\notin I(S_{\infty})\cup \{i\}}a_j(\hat{S}_{\bt})\right) \text{ (as } \frac{\tau}{|I(S_\infty)|}<\frac{1}{n-1})\\
		&\leq1 \text{ (as } a_j(\hat{S}_{\bt})\leqslant 1. \text{ for all } j\notin I(S_\infty\cup \{i\}).
	\end{align*}
	Since $\bt$ is arbitrary, it follows that $a_{i}(\hat{S}_{t})=\frac{\tau}{r_i(\hat{S}_{t})}$ for all $t\geqslant t_0$ with $\uti_t=i$. Hence, 
	\begin{align}
		a_{i}(\hat{S}_{t})&=\frac{\tau}{r_i(\hat{S}_{t})}\nonumber\\
		&=\frac{\tau}{|I(\hat{S}_{t})|}  \left(|I(\hat{S}_{t})|+\sum_{j\notin I(\hat{S}_{t})\cup \{i\}}a_j(\hat{S}_{t})\right)\nonumber\\
		&=\frac{\tau}{|I(S_\infty)|}  \left(|I(S_\infty)|+\sum_{j\notin I(S_\infty)\cup \{i\}}a_j(\hat{S}_{t})\right).\label{1}
	\end{align}
	Taking limit on both the sides of \ref{1}, we have 
	\begin{align*}
		&\gamma=\frac{\tau}{|I(S_\infty)|}  \left(|I(S_\infty)|+\sum_{j\notin I(S_\infty)\cup \{i\}}\gamma\right)\\
		\implies& \gamma=\frac{\tau|I(S_\infty)|}{(1+\tau)|I(S_\infty)|-\tau(n-1)}\\
		\implies& \gamma< \frac{\tau|I(S_\infty)|}{(1+\tau)|I(S_\infty)|-|I(S_\infty)|}=1.\\
	\end{align*}
	
	Now assume $(n-1)\tau\geqslant |I(S_\infty)|$. We have to show that $\gamma=1$. Assume  $\gamma<1$. Consider $i\notin I(S_\infty)$. Since by Claim \ref{cl_3} in Lemma \ref{lem_1}, $a_i(S_t)$ is an increasing sequence for $t>t_0$, $\gamma<1$ implies $a_i(S_t)<1$ for all  $t>t_0$.  This means $a_i(\hat{S}_t)=\frac{\tau}{r_i(\hat{S}_t)}$ for $t>t_0$ with $\uti_t=i$. Therefore, using similar arguments as before we have
	\begin{align*}
		&\gamma=\frac{\tau|I(S_\infty)|}{(1+\tau)|I(S_\infty)|-\tau(n-1)}\\
		\implies& \gamma\geqslant  \frac{\tau|I(S_\infty)|}{(1+\tau)|I(S_\infty)|-|I(S_\infty)|}=1.\\
		\
	\end{align*}
	But this is a contradiction to $\gamma<1$. Therefore, $\gamma=1$.
\end{proof}

\section{Proof of Theorem \ref{thm_ac_1}, Theorem\ref{thm_ac_1.5}, Theorem \ref{thm_ac_3}, and Theorem \ref{thm_ac_4}}\label{sec:appendix_E}
  \subsection{Proof of Theorem \ref{thm_ac_1}}

\begin{proof}
First assume $\tau\geqslant \frac{1}{n-1}$.	We first explore the limiting actions for a fixed agent sequence, and then we use this to find the limiting probability distribution.  Let $\uti$ be an agent sequence and $S$ be the DVSP induced by $\uti$. Note that by Remark \ref{rem_3}, it is enough to assume $\uti\in N_\infty$. Therefore, by Lemma \ref{ac_2}, all the agents outside $I(S_\infty)$ have the same action limit, and all the agents in $I(S_\infty)$ have the action limit 1. Let us denote the common limit by $\gamma$. 
We distinguish two cases based on the value of $N_1(\uti)$ (as in the proof of Theorem \ref{theo_1}) to find $\gamma$. Note that by the assumption of the theorem $\alpha\leqslant n-1$.\\  
 \noindent\textbf{Case 1:} $|N_1(\uti)|\geqslant \alpha$.\\
Recall that for this case the final infected set is $\{1\}$. Moreover, by the assumption of the theorem, $\tau(n-1)\geqslant 1$. Therefore, by Lemma \ref{lem_2}, $\gamma=1$. Hence, $a_{N}(S_\infty)=\vec{1}$.

\noindent\textbf{Case 2:} $|N_1(\uti)|\leqslant \alpha-1$.\\
 Recall that for this case, the final infected set is $N_1(\uti)\cup \{1\}$.  Therefore, by Lemma \ref{lem_2}, if $(n-1)\tau\geqslant |N_1(\uti)|+1$ then $a_{N}(S_\infty)=\vec{1}$, and if $(n-1)\tau<|N_1(\uti)|+1$ then
 
\[
 a_i(S_\infty)=
\begin{cases}
1 & \text{if } i \in I(S_{\infty}),\\
\frac{\tau(|N_1(\uti)|+1)}{(1+\tau)(|N_1(\uti)|+1)-\tau(n-1)} & \text{if } i \notin I(S_{\infty}).
\end{cases}
\]


Recall that $\beta=\min \{\lfloor(n-1)\tau \rfloor+1, \alpha+1\}$. Hence, combining Cases 1 and 2, we have the following:
\begin{enumerate}
	\item    $|N_1(\uti)|+1\in [\beta,\alpha]$ implies  
	
	\[
 a_i(S_\infty)=
\begin{cases}
1 & \text{if } i \in I(S_{\infty}),\\
\frac{\tau(|N_1(\uti)|+1)}{(1+\tau)(|N_1(\uti)|+1)-\tau(n-1)} & \text{if } i \notin I(S_{\infty}).
\end{cases}
\]
	
\item $|N_1(\uti)|+1\in [1,\beta-1]\cup [\alpha+1,n]$ implies $a_{N}(S_\infty)=\vec{1}$. 

\end{enumerate}

Note that (i) implies $a_{N}(S_\infty)\in A_{(|N_1(\uti)|+1)}$ when $|N_1(\uti)|+1\in [\beta,\alpha]$. Also, as $\mathbb{P}$ is uniform, any two vectors in $A_m$, for  $m\in [\beta,\alpha]$, have the same probability. Thus, we have the following distribution
 
 		\[ 
  		\mathbb{P}(a_{N}(S_\infty)=\vec{x})=
  	\begin{cases}
  1-\frac{\alpha-\beta+1}{n}& \text{ if }  \vec{x}=\vec{1}, \\ \\
  	\frac{1}{n \times \Mycomb[n-1]{m-1}} &\text{ if } \vec{x}\in A_m \text{ for some } m\in [\beta,\alpha], \\ \\
    	0 & \text{ otherwise. }
  	  	\end{cases}
    	\] 

 



		Now assume $\tau<\frac{1}{n-1}$. We follow the same structure as in the previous case, that is, we first explore the limiting actions for a fixed agent sequence, and then we use this to find the limiting probability distribution.  Let $\uti$ be an agent sequence and $S$ be the DVSP induced by $\uti$. Note that by Remark \ref{rem_3}, it is enough to assume $\uti\in N_\infty$. Therefore, by Lemma \ref{ac_2}, all the agents outside $I(S_\infty)$ have the same action limit, and all the agents in $I(S_\infty)$ have the action limit 1. Let us denote the common limit by $\gamma$. As by the assumption of the theorem, $(n-1)\tau<1$, we have $\alpha=n$, and hence, $|N_1(\uti)|\leqslant \alpha-1$. Moreover, for $|N_1(\uti)|\leqslant \alpha-1$ (shown in the proof of Theorem \ref{theo_1}), the final infected set is $N_1(\uti)\cup \{1\}$. Thus, $|I(S_\infty)|>(n-1)\tau$. 
Hence, by Lemma \ref{lem_2} if $|N_1(\uti)|+1<n$

\[
 a_i(S_\infty)=
\begin{cases}
1 & \text{if } i \in I(S_{\infty}),\\
\frac{\tau(|N_1(\uti)|+1)}{(1+\tau)(|N_1(\uti)|+1)-\tau(n-1)} & \text{if } i \notin I(S_{\infty}).
\end{cases}
\]


 and if $|N_1(\uti)|+1=n$, $a_{N}(S_\infty)=1$.
	Recall the notation $A_m$. By the above arguments, we have $a_{N}(S_\infty)\in A_{[|N_1(\uti)|+1]}$ when $|N_1(\uti)|+1<n$. Moreover, as $\mathbb{P}$ is uniform, any two vectors in $A_m$, for  $m\in [1,(n-1)]$, have the same probability. Thus, by Theorem \ref{theo_1}, we have the following distribution
	
 		\[ 
  		\mathbb{P}(a_{N}(S_\infty)=\vec{x})=
  	\begin{cases}
 \frac{1}{n}& \text{ if }  \vec{x}=\vec{1}, \\ \\
  	\frac{1}{n \times \Mycomb[n-1]{m-1}} &\text{ if } \vec{x}\in A_m \text{ for some } m\in [1,n-1], \\ \\
    	0 & \text{ otherwise.}
  	
  	\end{cases}
\]	
	
\end{proof}

\subsection{Proof of Theorem \ref{thm_ac_1.5}}

\begin{proof}
	We first explore the limiting actions for a fixed agent sequence, and then we use this to find the limiting probability distribution.  Let $\uti$ be an agent sequence and $S$ be the DVSP induced by $\uti$. Note that by Remark \ref{rem_3}, it is enough to assume $\uti\in N_\infty$. Therefore, by Lemma \ref{ac_2}, all the agents outside $I(S_\infty)$ have the same action limit, and all the agents in $I(S_\infty)$ have the action limit 1. Let us denote the common limit by $\gamma$. First assume $\tau\geqslant \frac{1}{n-1}$. By Theorem \ref{theo_1.5}, $I(S_\infty)=\{1\}$. Therefore, $(n-1)\tau\geqslant |I(S_\infty)|$, and hence, by Lemma \ref{2}, $\gamma=1$. Thus, $a_{N}(S_\infty)=\vec{1}$.

 Now assume that $\tau<\frac{1}{n-1}$. We distinguish two cases based on the value of $N_1(\uti)$ (as in the proof of Theorem \ref{theo_1.5}) to find $\gamma$. \\  
 \noindent\textbf{Case 1:} $|N_1(\uti)|=0$.\\
Recall that for this case the final infected set is $N$. Hence, $a_{N}(S_\infty)=\vec{1}$.

\noindent\textbf{Case 2:} $|N_1(\uti)|\geqslant 1$.\\
 Recall that for this case, the final infected set has  cardinality either $n$ or $n-1$. If cardinality is $n$ then $a_{N}(S_\infty)=\vec{1}$. If cardinality is $n-1$, then as $(n-1)\tau <1$, by Lemma \ref{2}, $\gamma=\tau$. Hence,
 
\[
 a_i(S_\infty)=
\begin{cases}
1 & \text{if } i \in I(S_{\infty}),\\
\tau & \text{if } i \notin I(S_{\infty}).
\end{cases}
\]
Note that this implies $a_{N}(S_\infty)\in A_{n-1}$. Also, as $\mathbb{P}$ is uniform, any two vectors in $A_{n-1}$ have the same probability. Thus, by Theorem \ref{theo_1.5}, we have the following distribution
 		\[ 
  		\mathbb{P}(a_{N}(S_\infty)=\vec{x})=
  	\begin{cases}
   1-\frac{n-1}{n^2} &\text{ if } \vec{x}=\vec{1}, \\ \\
  \frac{1}{n^2}& \text{ if }  \vec{x}\in A_{n-1}, \\ \\
  	0 & \text{ otherwise. }
  	
  	\end{cases}
\] 
\end{proof}

 \subsection{Proof of Theorem \ref{thm_ac_3}}

\begin{proof}
	We first explore the limiting actions for a fixed agent sequence, and then we use this to find the limiting probability distribution.  Let $\uti$ be an agent sequence and $S$ be the DVSP induced by $\uti$. Note that by Remark \ref{rem_3}, it is enough to assume $\uti\in N_\infty$. Therefore, by Lemma \ref{ac_2}, all the agents outside $I(S_\infty)$ have the same action limit, and all the agents in $I(S_\infty)$ have the action limit 1. Let us denote the common limit by $\gamma$. 
	We distinguish two cases based on the value of $N_1(\uti)$ (as in the proof of Theorem \ref{theo_2}) to find $\gamma$.\\  
 \noindent\textbf{Case 1:} $|N_1(\uti)|\geqslant \halpha$.\\
Recall that for this case the final infected set is $\{1\}$. Moreover, by the assumption of the theorem, $(n-1)\tau\geqslant 1$. Therefore, by Lemma \ref{lem_2}, $\gamma=1$. Hence, $a_{N}(S_\infty)=\vec{1}$.

\noindent\textbf{Case 2:} $|N_1(\uti)|\leqslant \halpha-1$.\\
 Recall that for this case, the final infected set is $N_1(\uti)\cup \{1\}$. Note that as $\halpha\leqslant n-1$, $N_1(\uti)\cup \{1\}\leqslant n-1$. Therefore, by Lemma \ref{lem_2}, if $(n-1)\tau\geqslant |N_1(\uti)|+1$ then $a_{N}(S_\infty)=\vec{1}$, and if $(n-1)\tau<|N_1(\uti)|+1$ then
 
\[
 a_i(S_\infty)=
\begin{cases}
1 & \text{if } i \in I(S_{\infty}),\\
\frac{\tau(|N_1(\uti)|+1)}{(1+\tau)(|N_1(\uti)|+1)-\tau(n-1)} & \text{if } i \notin I(S_{\infty}).
\end{cases}
\]


Recall that $\hbeta=\min \{\lfloor(n-1)\tau \rfloor+1,\halpha+1\}$. Thus, combining Cases 1 and 2, we have the following:
\begin{enumerate}
	\item    $|N_1(\uti)|+1\in [\hbeta,\halpha]$ implies
	
	\[
 a_i(S_\infty)=
\begin{cases}
1 & \text{if } i \in I(S_{\infty}),\\
\frac{\tau(|N_1(\uti)|+1)}{(1+\tau)(|N_1(\uti)|+1)-\tau(n-1)} & \text{if } i \notin I(S_{\infty}).
\end{cases}
\]

\end{enumerate}

\noindent Note that (i) implies $a_{N}(S_\infty)\in A_{[|N_1(\uti)|+1]}$ when $|N_1(\uti)|+1\in [\hbeta,\halpha]$. Also, as $\mathbb{P}$ is uniform, any two vectors in $A_m$, for  $m\in [\hbeta,\halpha]$, have the same probability. Thus, we have the following distribution

 		\[ 
  		\mathbb{P}(a_{N}(S_\infty)=\vec{x})=
  	\begin{cases}
  1-\frac{\halpha-\hbeta+1}{n}& \text{ if }  \vec{x}=\vec{1}, \\ \\
  	\frac{1}{n \times \Mycomb[n-1]{m-1}} &\text{ if } \vec{x}\in A_m \text{ for some } m\in [\hbeta,\halpha], \\ \\
    	0 & \text{ otherwise. }
  	\end{cases}
  	  	\]

\end{proof}

 \subsection{Proof of Theorem \ref{thm_ac_4}}
 
 \begin{proof}
 	We first explore the limiting actions for a fixed agent sequence, and then we use this to find the limiting probability distribution.  Let $\uti$ be an agent sequence and $S$ be the DVSP induced by $\uti$. Note that by Remark \ref{rem_3}, it is enough to assume $\uti\in N_\infty$. Therefore, by Lemma \ref{ac_2}, all the agents outside $I(S_\infty)$ have the same action limit, and all the agents in $I(S_\infty)$ have the action limit 1. Let us denote the common limit by $\gamma$. First assume $\talpha+1 < \balpha$. As shown in the proof of Theorem \ref{theo_3}, the infected set is either $\{1\}$ or $|N_1(\uti)|+1$ where $N_1(\uti)\in [1,\talpha]$. Since by the assumption of the theorem, $(n-1)\tau\geqslant 1$, we have $(n-1)\tau\geqslant I(S_\infty)$ when the infected set is $\{1\}$. Therefore, by Lemma \ref{2}, $\gamma=1$ and hence, $a_{N}(S_\infty)=\vec{1}$. On the other hand, if the final infected set is $N_1(\uti)\cup \{1\}$, the limiting action depends on $|N_1(\uti)|$. By Lemma \ref{lem_2}, if $(n-1)\tau\geqslant |N_1(\uti)|+1$ then $a_{N}(S_\infty)=\vec{1}$, and if $(n-1)\tau<|N_1(\uti)|+1$ then
 	
 \[
 a_i(S_\infty)=
\begin{cases}
1 & \text{if } i \in I(S_{\infty}),\\
\frac{\tau(|N_1(\uti)|+1)}{(1+\tau)(|N_1(\uti)|+1)-\tau(n-1)} & \text{if } i \notin I(S_{\infty}).
\end{cases}
\]

    Recall that the following was shown in the proof of Theorem \ref{theo_3} when $\talpha+1 \leqslant \balpha$:
    \begin{itemize}
    	\item $|I(S_\infty)|=1$ if $|N_1(\uti)|\in \{0,\talpha,\talpha+1,\ldots,n-1\}$ and
    	\item $|I(S_\infty)|=|N_1(\uti)|+1$ if $|N_1(\uti)|\in \{1,2,\ldots,\talpha-1\}$.
    \end{itemize}

Recall that $\tbeta=\min \{\lfloor(n-1)\tau \rfloor+1,\talpha+1\}$. 	Therefore, we have the following:
 	\begin{enumerate}
 		\item    $|N_1(\uti)|+1\in [\tbeta,\talpha]$ implies  
 		
 	\[
 a_i(S_\infty)=
\begin{cases}
1 & \text{if } i \in I(S_{\infty}),\\
\frac{\tau(|N_1(\uti)|+1)}{(1+\tau)(|N_1(\uti)|+1)-\tau(n-1)} & \text{if } i \notin I(S_{\infty}).
\end{cases}
\] 
	
 		\item $|N_1(\uti)|+1\in [1,\tbeta-1]\cup [\talpha+1,n]$ implies $a_{N}(S_\infty)=\vec{1}$. 
 		
 	\end{enumerate}
 	
 	Note that (i) implies $a_{N}(S_\infty)\in A_{[|N_1(\uti)|+1]}$ when $|N_1(\uti)|+1\in [\tbeta,\talpha]$. Also, as $\mathbb{P}$ is uniform, any two vectors in $A_m$, for  $m\in [\tbeta,\talpha]$, have the same probability. Thus, we have the following distribution
 	
 		\[ 
  		\mathbb{P}(a_{N}(S_\infty)=\vec{x})=
  	\begin{cases}
  1-\frac{\talpha-\tbeta+1}{n}& \text{ if }  \vec{x}=\vec{1}, \\ \\
  	\frac{1}{n \times \Mycomb[n-1]{m-1}} &\text{ if } \vec{x}\in A_m \text{ for some } m\in [\bbeta,\balpha], \\ \\
    	0 & \text{ otherwise. }
  	  	\end{cases}
  	  	\]

 
  	Now assume $2\leqslant \balpha< \talpha+1$.  
  	Recall that the following was shown in the proof of Theorem \ref{theo_3} when $2\leqslant \balpha< \talpha+1$:
  		\begin{enumerate}[(i)]
  		\item $|I(S_\infty)|=1$ if $|N_1(\uti)|\in \{0,\talpha,\talpha+1,\ldots,n-1\}$,
  		\item $|I(S_\infty)|=|N_1(\uti)|+1$ if $|N_1(\uti)|\in \{1,2,\ldots,\balpha-2\}$,
  		\item $|I(S_\infty)|=n$ if $|N_1(\uti)|\in \{\balpha-1,\ldots,\talpha-1\}$ and there is no $i\in N$ such that $k_i=k_1+1$ and $\uti_{k_1+2}=i$, and
  		\item $|I(S_\infty)|=n-1$ if $|N_1(\uti)|\in  \{\balpha-1,\ldots,\talpha-1\}$ and there is $i\in N$ such that $k_i=k_1+1$ and $\uti_{k_1+2}=i$.
  	\end{enumerate}
  	By the assumption of the theorem, $(n-1)\tau\geqslant 1$ and $\tau<1$. Thus, if $|I(S_\infty)|=1$ we have $(n-1)\tau\geqslant |I(S_\infty)|$, and if $|I(S_\infty)|= (n-1)$ we have $(n-1)\tau< |I(S_\infty)|$. Recall that $\bbeta=\min \{\lfloor(n-1)\tau \rfloor+1,\balpha\}$. Combining all these observations, we may write the following
  	
  	\begin{enumerate}
  		\item    $|I(S_\infty)|\in [\bbeta,\balpha-1]\cup \{n-1\}$ implies  
  		
  		\[
 a_i(S_\infty)=
\begin{cases}
1 & \text{if } i \in I(S_{\infty}),\\
\frac{\tau(|I(S_{\infty})|)}{(1+\tau)(|I(S_{\infty})|)-\tau(n-1)} & \text{if } i \notin I(S_{\infty}).
\end{cases}
\] 
  		
  		
  		\item $|I(S_\infty)|\in [1,\bbeta-1]\cup \{n\}$ implies $a_{N}(S_\infty)=\vec{1}$. 
  \end{enumerate}
  	Note that (i) implies $a_{N}(S_\infty)\in A_{(|I(S_\infty)|)}$ when $|I(S_\infty)|\in [\bbeta,\balpha-1]\cup \{n-1\}$. Also, as $\mathbb{P}$ is uniform, any two vectors in $A_m$, for  $m\in [\bbeta,\balpha-1]\cup \{n-1\}$, have the same probability. Therefore, using Theorem \ref{theo_3}, we have the following distribution
  	
  	\[ 
  		\mathbb{P}(a_{N}(S_\infty)=\vec{x})=
  	\begin{cases}
  	1+\frac{\bbeta-\balpha}{n}-\eta(\talpha,\balpha,n) & \text{ if }  \vec{x}=\vec{1}, \\ \\
  	\frac{1}{n \times \Mycomb[n-1]{m-1}} &\text{ if } \vec{x}\in A_m \text{ for some } m\in [\bbeta,\balpha-1], \\ \\
  	\frac{\eta(\talpha,\balpha,n)}{n-1} &\text{if } \vec{x}\in A_{n-1}, \\ \\
  	0 & \text{ otherwise. }
  	  	\end{cases}
  	  	\]

  	\end{proof}

   \subsection{Proof of Theorem \ref{thm_ac_5}}

\begin{proof}
	We first explore the limiting actions for a fixed agent sequence, and then we use this to find the limiting probability distribution.  Let $\uti$ be an agent sequence and $S$ be the DVSP induced by $\uti$. Note that by Remark \ref{rem_3}, it is enough to assume $\uti\in N_\infty$. Therefore, by Lemma \ref{ac_2}, all the agents outside $I(S_\infty)$ have the same action limit, and all the agents in $I(S_\infty)$ have the action limit 1. Let us denote the common limit by $\gamma$. First assume $\tau= \frac{a}{n-1}$. By Theorem \ref{theo_4}, $I(S_\infty)=N$. Therefore,  by Lemma \ref{ac_2}, $a_{N}(S_\infty)=\vec{1}$.
	
	Now assume that $\tau<\frac{a}{n-1}$. We distinguish two cases based on the value of $N_1(\uti)$ (as in the proof of Theorem \ref{theo_4}) to find $\gamma$. \\  
	\noindent\textbf{Case 1:} $|N_1(\uti)|=0$.\\
	Recall that for this case the final infected set is $N$. Hence, $a_{N}(S_\infty)=\vec{1}$.
	
	\noindent\textbf{Case 2:} $|N_1(\uti)|\geqslant 1$.\\
	Recall that for this case, the final infected set has  cardinality either $n$ or $n-1$. If cardinality is $n$ then $a_{N}(S_\infty)=\vec{1}$. If cardinality is $n-1$, then as $(n-1)\tau <1$, by Lemma \ref{2}, $\gamma=\tau$. Hence,
	
	\[
	a_i(S_\infty)=
	\begin{cases}
		1 & \text{if } i \in I(S_{\infty}),\\
		\tau & \text{if } i \notin I(S_{\infty}).
	\end{cases}
	\]
	Note that this implies $a_{N}(S_\infty)\in A_{n-1}$. Also, as $\mathbb{P}$ is uniform, any two vectors in $A_{n-1}$ have the same probability. Thus, by Theorem \ref{theo_4}, we have the following distribution
	\[ 
	\mathbb{P}(a_{N}(S_\infty)=\vec{x})=
	\begin{cases}
		1-\frac{n-1}{n^2} &\text{ if } \vec{x}=\vec{1}, \\ \\
		\frac{1}{n^2}& \text{ if }  \vec{x}\in A_{n-1}, \\ \\
		0 & \text{ otherwise. }
		
	\end{cases}
	\] 
\end{proof}

\end{appendix}


\end{document}